\documentclass[12pt]{article}
\pdfoutput=1

\usepackage[a4paper, margin=1.0in]{geometry}
\usepackage{color,enumerate,graphicx,tikz}
\usepackage{multicol}
\usepackage[latin1]{inputenc}
\usepackage{amsmath,amssymb,mathtools,amsthm}
\usepackage{bm}
\usepackage{style,dsfont} 
\usepackage{algpseudocode,algorithm,algorithmicx} 
\usepackage{graphicx}
\usepackage{caption}
\usepackage[font=footnotesize,position=b]{subcaption}

\numberwithin{equation}{section}
\title{A tumor growth model of Hele-Shaw type\\ as a gradient flow}
\author{
Simone Di Marino \thanks{Lab. des Math., Universit\'e Paris-Sud, Orsay. } \thanks{Istituto Nazionale di Alta Matematica, Sede SNS Pisa, email: \textsf{simone.dimarino@altamatematica.it}} ,  L\'ena\"{i}c Chizat \thanks{Universit\'e Paris-Dauphine, PSL Research University, CNRS, CEREMADE, 75016 Paris, France}
   }

\begin{document}

\maketitle

\abstract{
%
%
In this paper, we characterize a degenerate PDE as the gradient flow in the space of nonnegative measures endowed with an optimal \emph{transport-growth} metric. The PDE of concern, of Hele-Shaw type, was introduced by Perthame \emph{et}.\ \emph{al}.\ as a mechanical model for tumor growth and the  metric was introduced recently in several articles as the analogue of the Wasserstein metric for nonnegative measures. %
%
We show existence of solutions using \emph{minimizing movements} and show uniqueness of solutions on convex domains by proving the \emph{Evolutional Variational Inequality}. Our analysis does not require any regularity assumption on the initial condition. We also derive a numerical scheme based on the discretization of the gradient flow and the idea of entropic regularization. We assess the convergence of the scheme on explicit solutions.
%
In doing this analysis, we prove several new properties of the optimal transport-growth metric, which generally have a known counterpart for the Wasserstein metric.
}

\section{Introduction}

\subsection{Motivation}
Modeling tumor growth is a longstanding activity in applied mathematics that has become a valuable tool for understanding cancer developement.
%
At the macroscopic and continuous level, there are two main categories of models: the cell density models -- which describe the tumor as a density of cells which evolve in time -- and the free boundary models -- which describe the evolution of the domain conquered by the tumor by specifying the geometric motion of its boundary.
Perthame et al.~\cite{perthame2014hele,mellet2015hele} have exhibited connection between these two approaches: by taking the incompressible limit of a standard density model of growth/diffusion, one recovers a free boundary model of Hele-Shaw type.  

More precisely, they consider a monophasic density of cells $\rho(x,t)$ (with $x\in \R^{d}$ the space variable and $t\geq 0$ the time) whose motion is driven by a scalar pressure field $p(x,t)$ through Darcy's law and which grows according to the rate of growth which is modeled as a function of the pressure $\Phi(p(x,t))$ where $\Phi$ is continuously decreasing and null for $p$ greater than a so-called ``homeostatic'' pressure. The equation of evolution for $\rho$ is then
\begin{equation}\label{eq:mechanicalmodel}
\left\{
\begin{aligned}
    \partial_t \rho  - \nabla \cdot (\rho \nabla p) &= \Phi(p) \rho\, , \text{ for $t>0$}
\\
p &= \rho^m, \quad \text{with } m\geq1\, ,
\\
\rho(0,\cdot) &= \rho_0 \in L^1_+(\Omega)
\end{aligned}
\right.
\end{equation}
where the relation $p=\rho^{m}$ accounts for a slow-diffusive motion. For suitable initial conditions, they show that when $m$ tends to infinity -- the so-called \emph{stiff} or \emph{incompressible} or \emph{hard congestion} limit -- the sequence of solutions $(\rho^{m},p^{m})$ of \eqref{eq:mechanicalmodel} tends to a limit $(\rho^{\infty},p^{\infty})$ satisfying a system of the form \eqref{eq:mechanicalmodel} where the relation between $\rho$ and $p$ is replaced by the \emph{Hele-Shaw} graph constraint $p(1-\rho)=0$.

Our purpose is to study directly this stiff limit system from a novel mathematical viewpoint, focusing on the case of a rate of growth depending linearly on the pressure $\Phi(p)=4(\lambda-p)_{+}$, with a homeostatic pressure $\lambda>0$. In a nutshell, we show that the stiff limit system
\begin{equation}\label{eq:mainPDE}
\begin{cases}
    \partial_t \rho - \nabla \cdot (\rho \nabla p) = 4(\lambda - p)_{+}\rho 
\\
p(1-\rho)=0
\\
0 \leq \rho \leq 1
\\
\rho(0,\cdot) = \rho_0
\end{cases}
\end{equation}
characterizes gradient flows of the functional $G: \mathcal{M}_{+}(\Omega) \to \R \cup \{+\infty\}$ defined as, with $\mathcal{L}^d$ the Lebesgue measure on $\R^d$,
\begin{equation}\label{eq:functional}
G(\rho) = \begin{cases}
-\lambda \rho(\Omega) & \text{if $\rho \ll \mathcal{L}^d$ and $\tfrac{d \rho}{d \mathcal{L}^d}\leq 1$,} \\
+\infty & \text{otherwise}
\end{cases}
\end{equation}
in the space of nonnegative measures $\mathcal{M}_{+}(\Omega)$ endowed with a metric which accounts for the displacement of mass and the growth/shrinkage which is necessary to interpolate between two measures. 

This approach has the following advantages:
\begin{itemize}
\item on a qualitative level, it gives a simple interpretation of the Hele-Shaw tumor growth model. Namely,  \eqref{eq:mainPDE} describes \emph{the most efficient way for a tumor to gain mass under a maximum density constraint}, where efficience means small displacement and small rate of growth.
\item on a theoretical level, we show existence of solutions to \eqref{eq:mainPDE}  without regularity assumption on the initial condition (unlike \cite{mellet2015hele}) and uniqueness on compact convex domains. Also, our study showcases another application of the theory of gradient flows in metric spaces, beyond Wasserstein spaces.
\item on a numerical level, relying on recent advances on algorithms for unbalanced optimal transport problems \cite{chizat2016scaling}, the gradient flow approach allows for a simple numerical scheme for computing solutions to \eqref{eq:mainPDE}.
\end{itemize}

\subsection{Background and main result}
In order to make precise statements, let us define what is meant by \emph{gradient flow} and by \emph{transport-growth metric} in this article. In $\R^{d}$, the gradient flow of a function $G:\R^{d}\to \{+\infty\}$ is a continuous curve $x:\R_{+}\to \R^{d}$  which is solution to the Cauchy problem
\begin{equation}\label{eq:euclideanGF}
\left\{ \begin{aligned}
\tfrac{d}{dt}x(t) &= - \nabla G(x(t)), \text{ for $t>0$} \\
x(0) &= x_{0}\in \R^{d}\, .
\end{aligned} \right.
\end{equation}
However, in a (non-Riemannian) metric space $(X,d)$, the gradient $\nabla G$ of a functional $G:X\to \R \cup \{+\infty\}$ is not defined anymore. Yet, several extensions of the notion of gradient flows exist, relying on the variational structure of \eqref{eq:euclideanGF}, see \cite{ambrosio2008gradient} for a general theory. One approach is that of \emph{minimizing movements} introduced by De Giorgi, and originates from the discretization in time of \eqref{eq:euclideanGF} through the implicit Euler scheme : starting from $\tilde{x}^{\tau}_{0} = x_{0} \in X$ define a sequence of points $(\tilde{x}^{\tau}_{k})_{k\in \mathbb{N}}$ as follows
\begin{equation}\label{eq:metricEulerScheme}
\tilde{x}_{k+1} \in \argmin_{x\in X} \left\{ G(x) + \tfrac{1}{2\tau} d(x,\tilde{x}^{\tau}_{k})^{2} \right\}\, .
\end{equation}
By suitably interpolating this discrete sequence, and making the time step $\tau$ tend to $0$, we recover a curve which, in a Euclidean setting, is a solution to \eqref{eq:euclideanGF}. This leads to the following definition.
\begin{definition}[Uniform minimizing movements]\label{def:gradientflow}
Let $(X,d)$ be a metric space, $G:X\to \R\cup\{\infty\}$ be a functional and $x_{0}\in X$. A curve $x:\R_{+}\to X$ is a uniform minimizing movement if it is the pointwise limit as of a sequence of curves $x^{\tau_i}$ defined as $x^{\tau_i}(t) = \tilde{x}^{\tau_i}_{k} $ for $t\in [k\tau_i,(k+1)\tau_i [$ for some sequence generated by \eqref{eq:metricEulerScheme}, with $\tau_i \to 0$.
\end{definition}
When the metric space is the space of probability measures endowed with an optimal transport metric $(\mathcal{P}(\Omega),W_{2})$, this time discretization is known as the \emph{JKO scheme}. It is named after the authors of the seminal paper~\cite{jko} where it is used to recover (in particular) the heat equation by taking the uniform minimizing movement of the entropy functional. 
A more precise and more restrictive notion of gradient flow is given by the \emph{evolutional variational inequality} (EVI).

\begin{definition}[EVI gradient flow]
An absolutely continuous curve $(x(t))_{t\in [0,T]}$ in a metric space $(X,d)$ is said to be an $\mathrm{EVI}_\alpha$ (for $\alpha \in \R$) solution of gradient flow of $G:X\to \R \cup \{\infty\}$ if for all $y\in \dom G$ and a.e.\ $t\in {]0,T[}$ it holds
\[
\frac12 \frac{d}{dt}d(x(t),y)^2 \leq F(y) - F(x(t)) -\frac{\alpha}{2}d(x(t),y)^2.
\]
\end{definition}

This definition is in fact more restrictive because $\mathrm{EVI}_\alpha$ implies uniqueness, but also $\alpha$-convexity of the functional $F$.

In this article, we endow the space of nonnegative measures $\mathcal{M}_{+}(\Omega)$ on a domain $\Omega \subset \R^{d}$ with another metric structure which has been introduced recently by several teams~\cite{chizat2015interpolating,liero2015optimal,kondratyev2015new} and called ``Kantorovich-Fisher-Rao'' or  ``Wasserstein-Fisher-Rao'' or ``Hellinger-Kantorovich'' in these various works. Here we choose to simply use the notation $\dist$ and refer to it as the optimal transport-growth metric. The simplest way to understand this metric is probably through a Riemannian metric point of view: formally, its metric tensor is an inf-convolution between the tensor of the Wasserstein metric and the tensor of the Fisher-Rao metric. Indeed, the distance between two nonnegative measures $\mu$ and $\nu$ can be computed by finding an interpolating curve $[0,1] \ni t\mapsto \rho_{t} \in \mathcal{M}_{+}(\Omega)$, such that $\rho_{0} = \mu$ and $\rho_{1} = \nu$, of minimal length according to these metric tensors, i.e $\dist(\mu,\nu)$ is the square root of
\begin{multline}\label{eq:WFRgeodesic}
\min_{\rho,v,\alpha} \Big\{ \int_{0}^{1} (\Vert v_{t} \Vert^{2}_{L^{2}(\rho_{t})}+\tfrac14 \Vert \alpha_{t} \Vert^{2}_{L^{2}(\rho_{t})})dt \; : \; (v_t,\alpha_t) \in L^2(\rho_t)^d\times L^2(\rho_t), \\
\partial_{t}\rho_{t} = -\nabla \cdot (\rho_{t}v_{t}) + \alpha_{t} \rho_{t}   \text{ weakly}
 \text{ and } (\rho_{0},\rho_{1}) = (\mu,\nu) \Big\}
\end{multline}
and any optimal interpolation $\rho$ is a geodesic for this metric (some explicit geodesics are studied in \cite{liero2015geodesics}). Just as for the standard optimal transport problems, it is possible to formulate this metric in terms of optimal coupling problems \cite{chizat2015unbalanced,liero2015optimal}. We skip the derivation of those equivalences which are non-trivial and use, as a definition of $\dist$, the optimal entropy-transport problem~\cite{liero2015optimal}. This formulation involves the \emph{relative entropy} between two nonnegative measures (also known as Kullback-Leibler divergence) defined as
\begin{equation}\label{eqn:def_entropy}
\Ent{\nu}{\mu} \eqdef 
\begin{cases}
\int_{\Omega} (\sigma \log(\sigma) - \sigma + 1)d\mu  & \text{if $\nu\ll\mu$ and $\nu = \sigma \mu$} \\
+ \infty & \text{otherwise.} 
\end{cases}
\end{equation}
\begin{definition}[Optimal transport-growth metric]
Let $(\mu, \nu)\in \mathcal{M}_{+}(\Omega)^{2}$ be two measures on a domain $\Omega\subset \R^{d}$. The metric $\dist$ is defined as
\begin{equation}\label{eq:KFdef}
\dist(\mu_1,\mu_2) \eqdef \left\{ \min_{\gamma \in \mathcal{M}_+(\Omega^2)} \int_{\Omega\times \Omega} \clog(x,y) d \gamma+ \Ent{\gamma_1}{\mu_1} + \Ent{\gamma_2}{ \mu_2}  \right\}^{\frac12}
\end{equation}
where $\clog(x,y)\eqdef -\log \cos^2 (\min \{ |y-x|, \tfrac{\pi}{2} \})$ and $\gamma_1$ and $\gamma_2$ are the marginals of $\gamma$ on the first and second factors of the product space $\Omega\times \Omega$.
\end{definition}

\begin{proposition}[\cite{liero2015optimal}]\label{prop:metric properties}
If $\Omega$ is closed, then the space $(\mathcal{M}_+(\Omega),\dist)$ is a complete metric space.  Its topology is equivalent to the weak topology (in duality with continuous bounded functions).
\end{proposition}
The main result of this article, proved in Section \ref{sec:proof}, makes a link between the tumor growth model \eqref{eq:mainPDE}, the metric $\dist$ and the functional \eqref{eq:functional}. Note that we assume a definition of the distance $\dist$ that is based on the Euclidean metric on $\R^d$ and not the geodesic distance of $\Omega$, which differ when $\Omega$ is not convex.

\begin{maintheorem}\label{maintheorem}
Let $\Omega$ be an open bounded $H^1$-extension domain of $\R^{d}$ and $\rho_{0} \in L^1_+(\Omega)$ such that $\rho_0\leq 1$ and $T>0$. Then any minimizing movement with $G$ as in \eqref{eq:functional}, is a solution of \eqref{eq:mainPDE} on $[0,T]$ starting from $\rho_0$, with some $p\in L^2([0,T],H^1(\Omega))$. Moreover if $\Omega$ is convex we have that every solution of \eqref{eq:mainPDE} is an $\mathrm{EVI}_{(-2\lambda)}$ solution of gradient flow of $G$ in the metric space $(\mathcal{M}_+(\Omega), \dist)$. In particular in this case we have uniqueness for \eqref{eq:mainPDE}.
\end{maintheorem}

The existence result is stated in Proposition \ref{prop:gradientflowexistence} and the EVI characterization, with uniqueness, in Proposition \ref{prop:uniqueness}.

%

\begin{remark}
The concept of solutions to the system \eqref{eq:mainPDE} is understood in the weak sense, i.e.\ we say that the family of triplets $(\rho_t,v_t,g_t)_{t\geq0}$ is a solution to
$$\partial_t \rho_t -\nabla \cdot (\rho_t v_t) = g_t \rho_t$$
if for all $\phi\in C^{\infty}_c(\bar{\Omega})$, the function $t\mapsto \int_\Omega \phi(x) d\rho_t(x)$ is well defined, absolutely continuous on $[0,+\infty[$ and for a.e. $t\geq0$ we have
\[
\frac{d}{dt} \int_\Omega \phi d\rho_t = \int_\Omega (\nabla \phi \cdot v_t + \phi g_t)d\rho_t\, .
\]
This property implies that $t\mapsto \rho_t$ is weakly continuous, that the PDE is satisfied in the distributional sense, and imposes no-flux (a.k.a.\ Neumann) boundary conditions for $v_t$. Equation \eqref{eq:mainPDE} is a specialization of this equation with $v_t= -\nabla p_t$ and $g_t=4(\lambda-p_t)_+$.
\end{remark}

\subsection{Short informal derivation}
Before proving the result rigorously, let us present an informal discussion, inspired by \cite{santambrogio2017euclidean}, in order to grasp the intuition behind the result.
%
Stuying the optimality conditions in the dynamic formulation of $\widehat W_2$, one sees that the velocity and the growth fields are derived from a dual potential (proofs of this fact can be found in \cite{kondratyev2015new, liero2015optimal}) as $(v_t,g_t)=(\nabla \phi_t, 4 \phi_t)$ and one has 
\[
\widehat W_2^2(\mu,\nu) = \inf_{(\phi_t)_t} \left\{ \int_0^1 \int_\Omega (| \nabla \phi_t |^2 + 4 |\phi_t |^2) d \rho_t d t\;;\; \partial_t \rho_t = - \nabla \cdot (\nabla \phi_t \rho_t) + 4\phi_t \rho_t\right\}
\]
where $(\rho_t)_{t\in [0,1]}$ is a path that interpolates between $\mu$ and $\nu$. This suggests to interpret $\widehat W_2$ as a Riemannian metric with tangent vectors at a point $\rho\in \mathcal{M}_+(\Omega)$ of the form $\partial_t \rho = - \nabla \cdot (\nabla \phi \rho) + 4\phi\rho$ and the metric tensor 
\[
\langle \partial_t \rho_1,\partial_t \rho_2\rangle_\rho= \int_\Omega (\nabla \phi_1 \cdot \nabla \phi_2 + 4 \phi_1\cdot \phi_2) d \rho.
\]

Now consider a smooth functional by $F: \mathcal{M}_+(\Omega) \to \R$ and denote $F'$ the unique function such that $\frac{d}{d\epsilon} F (\rho + \epsilon \chi)\vert_{\epsilon = 0} = \int_{\Omega} F'(\rho) d \chi$ for all admissible perturbations $\chi\in \mathcal{M}(\Omega)$. Its gradient at a point $\rho$ satisfies for a tangent vector $\partial_t \rho = - \nabla \cdot ( \rho \nabla \phi) + 4 \phi \rho$, by integration by part,
\[
\langle \grad_\rho F, \partial_t \rho \rangle_\rho = \int_\Omega F'(\rho)\partial_t \rho 
= \int_\Omega (\nabla F'(\rho) \cdot \nabla \phi + 4 F'(\rho) \cdot \phi) \rho
\]
which shows that, by identification, one has
\begin{equation*}\label{eq:gradient flow smooth formal}
\grad_\rho F = - \nabla \cdot (\rho \nabla F'(\rho)) + 4 \rho F'(\rho).
\end{equation*}
Note that this formula shows that there is a strong relationship between the diffusion and the reaction terms for $\widehat{W}_2$-gradient flows.

Now consider the functional $F_{m}(\rho)=-\lambda \rho(\Omega)+\frac{1}{m+1}\rho^{m+1}$ ($\rho$ is identified with its Lebesgue density). The associated gradient flow is the diffusion-reaction system~\eqref{eq:mechanicalmodel} because $F'_m(\rho) = -\lambda +\rho^{m}$.  The functional $G$ introduced in \eqref{eq:functional} can be understood as the stiff limit as $m\to \infty$ of the sequence of functionals $F_m$. Theorem~\ref{maintheorem} expresses that the gradient flow structure is preserved in the limit $m\to \infty$ where one recovers the hard congestion model \eqref{eq:mainPDE}. The proof we propose follows however a different approach, and directly starts with the hard congestion model.



\subsection{Related work}
In the context of Wasserstein gradient flows, free boundary models have already been modeled in \cite{otto1998dynamics,giacomelli2001variatonal} where a thin plate model of Hele-Shaw type is recovered by minimizing the interface energy. More recently, crowd motions have been modeled with these tools in~\cite{maury2010macroscopic,maury2011handling} in a series of works pioneering the study of Wasserstein gradient flows with a hard congestion constraint.
The success of Wasserstein gradient flows in the field of PDEs has naturally led to generalizations of optimal transport metrics in order to deal with a wider class of evolution PDEs, such as the heat flow with Dirichlet boundary conditions \cite{figalli2010new}, and diffusion-reaction systems~\cite{liero2013gradient}. The specific metric $\dist$, has recently been used to study population dynamics~\cite{kondratyev2016fitness} and gradient flows structure for more generic smooth functionals have been explored in \cite{gallouet2016jko} where the author consider splitting strategy, i.e.\ they deal with the transport and the growth term in a alternative, desynchronized manner.

Our work was pursued simultaneously and independantly of  \cite{gallouet2017unbalanced} where this very class of tumor growth model are studied using tools from optimal transport. These two works use different approaches and are complementary: our focus is on the stiff models \eqref{eq:mechanicalmodel} and we directly study the incompressible system with specific tools while \cite{gallouet2017unbalanced} focuses primary on the diffusive models \eqref{eq:mechanicalmodel}, and recover stiff system by taking a double limit. Their approach is thus not directly based on a gradient flow, but is more flexible and allows to deal with nutrient systems.

\subsection{Organization of the paper}

In Section \ref{sec:proof}, we give the proof of Theorem~\ref{maintheorem}, which involves a number of preliminary results about entropy-transport problems and the metric $\dist$. In Section \ref{sec:numericalscheme}, we introduce a numerical scheme for solving \eqref{eq:mainPDE} based on the discretization of the gradient flow. We derive in Section~\ref{sec:sphere} the explicit solution for spherical initial condition, which allows to evaluate the precision of the numerical scheme. We conclude in Section~\ref{sec:illustrations} with numerical illustrations on 1-D and 2-D domains.

\subsection{Aknowledgements}

We want to thank the ANR project Mokaplan, whose seminars have been inspiration for this work, as well as FSMP, that gave the possibility to the second named author to pass a period in Paris which has been helpful in the redaction of the paper. Moreover we want to thank Giuseppe Savar\'e, Alexander Liero, Leonard Monsaingeont and Thomas Gallou\"{e}t for many useful discussions.

\section{Proof of the main result}
\label{sec:proof}
\subsection{Entropy-transport problems}

In this section we consider optimal entropy-transport problems associated to cost functions $c:\Omega^2\to \bar{\R}$, defined as
\begin{equation}\label{eqn:def_trent}
T_c(\mu_1, \mu_2) \eqdef \inf_{\gamma \in \mathcal{M}_+ ( \Omega \times \Omega )} \left\{  \int_{\R^{2d}} c(x,y)  \, d \gamma + \Ent {\gamma_1}{ \mu_1 } + \Ent {\gamma_2}{ \mu_2} \right\}
\end{equation} 
where $\gamma_1$ and $\gamma_2$ are the marginals of $\gamma$ on the factors of $\Omega \times \Omega$ and $\Entv$ is the relative entropy functional, defined in \eqref{eqn:def_entropy}. The main role is played by the cost 
\[
\clog(x,y) \eqdef -\log \cos^2(\min\{\vert y-x\vert,\pi/2\})
\]
for which one recovers the definition of $\dist$ in \eqref{eq:KFdef}. A family of Lipschitz costs $\cn$ approximating $\clog$ is also used. These costs are constructed from the following approximation argument for the function $f_\ell :  [ 0, \infty) \to [0, \infty]$ defined by
\begin{equation}\label{eq:functionfn}
f_\ell (t) \eqdef - \ln ( \cos^2 ( \min\{ t, \pi/2 \} ))\, .
\end{equation}
Its proof is postponed to the appendix.

\begin{lemma}\label{lem:f_n} The function $f_\ell$ is convex and satisfies $f_\ell'^2 = 4(e^{f_\ell}-1)$ in $[0, \pi/2)$. It can be approximated by an increasing sequence of strictly convex Lipschitz functions $f_n:[0, \infty) \to [0, \infty)$ such that
\begin{itemize}
\item[(i)] $0 \leq f_n \leq f_m \leq f$ for every $n \leq m$ and $f_n(t) \uparrow f (t)$ pointwise for every $t$; moreover $f_n(t)=f(t)$ for $t \in [0,1]$. 
\item[(ii)] For all $n$ we have that $f_n'^2 \leq 4(e^{f_n}-1)$ in $[0,\infty)$ and $ e^{f_n(t)}-1 \geq t^2$.
\end{itemize}
\end{lemma}
It follows from the definitions that $\clog(x,y)=f_\ell(|x-y|)$ and $\dist(\mu_1, \mu_2) = \sqrt{T_{\clog} ( \mu_1, \mu_2)}$. We also introduce the notation $\cn(x,y) = f_n( | x-y |)$.

The following characterization is the equivalent of the dual formulation of Kantorovich optimal transport in this setting and is proven in \cite[Thm. 4.14]{liero2015optimal} and corollaries.
\begin{theorem}\label{th:duality} Let us consider an $L$-Lipschitz cost $c \geq 0$. Then we have
$$ T_c(\mu_1, \mu_2)= \max_{\alpha, \beta \in {\rm Lip}_L(\Omega)} \left\{ \int (1- e^{-\alpha}) \, d \mu_1 + \int (1- e^{-\beta}) \, d \mu_2 \; : \; \alpha + \beta \leq c \right\}.$$
Here and in the following, the constraint has to be understood as $\alpha(x) +\beta (y) \leq c(x,y)$, for all $(x,y)\in \Omega^2$. Moreover, if $\gamma$ denotes a minimizer in the primal problem and $\alpha, \beta$ maximizers in the dual, we have the following compatibility conditions:
\begin{itemize}
\item[(i)]$\gamma_1 = e^{-\alpha} \mu_1$;
\item[(ii)] $ \gamma_2= e^{-\beta} \mu_2$;
\item[(iii)] $ \alpha(x)+ \beta(y) = c(x,y)$ for $\gamma$-a.e.\ $(x,y)$. In particular if $\mu_1$ is absolutely continuous with respect to the Lebesgue measure one has $ \nabla \alpha(x) = \partial_x c (x,y)$ for $\gamma$-a.e. $(x,y)$.
\item[(iv)] $T_c(\mu_1, \mu_2) = \mu_1(\R^d) + \mu_2(\R^d) - 2 \gamma(\R^d\times \R^d)$.
In particular we have that $\gamma_1$ and $\gamma_2$ are unique and $\alpha$ and $\beta$ are uniquely defined in the support of $\mu_1$ and $\mu_2$, respectively.

\end{itemize}
\end{theorem}

Some stability properties follow, both in term of the measures and of the costs.

\begin{proposition}\label{prop:uniform} Let us consider an $L$-Lipschitz cost $c \geq 0$. Then if $\mu_{n,i} \weakto \mu_i$ for $i=1,2$ and all the measures are supported on a bounded domain $\Omega$ then, denoting by $\alpha_n, \beta_n$ the maximizers in the dual problem, we have that  $\alpha_n \to \alpha$ and $\beta_n \to \beta$ locally uniformly where $\beta$ and $\alpha$ are maximizers in the dual problem for $\mu_1$ and $\mu_2$. Moreover $T_c(\mu_{n,1} ,\mu_{n,2}) \to T_c(\mu_1, \mu_2)$.
\end{proposition}

\begin{proof} First we show that $\lim_{n\to \infty} T_c(\mu_{n,1}, \mu_{n,2})=T_c( \mu_1,\mu_2)$. Let us consider $\{ \gamma_n \}$ the set of optimal plans in \eqref{eqn:def_trent} which forms a precompact set because the associated marginals do \cite[Prop. 2.10]{liero2015optimal}. Thus, from a subsequence of indices that achieves $\liminf_{n \to \infty} T_c(\mu_{n,1}, \mu_{n,2})$, one can again extract a subsequence for which the optimal plans weakly converge, to an \emph{a priori} suboptimal plan. Using the joint semicontinuity of the entropy and the continuity of the cost we deduce
$$\liminf_{n \to \infty} T_c(\mu_{n,1}, \mu_{n,2}) \geq T_c(\mu_1, \mu_2).$$
Similarly, the sequence of optimal dual variables $\alpha_n, \beta_n$ form a precompact set since it is a sequence of bounded $L$-Lipschitz functions (see \cite[Lem. 4.9]{liero2015optimal}). Any weak cluster point $\alpha_0,\beta_0$ satisfies $\alpha_0 + \beta_0 \leq c$ and in particular, taking again a subsequence of indices achieving the $\limsup$ and $\alpha_0,\beta_0$ a cluster point of it, we have
\begin{align*} \limsup_{n \to \infty} T_c(\mu_{n,1}, \mu_{n,2}) &= \limsup_{n \to \infty} \int_{\Omega} (1- e^{-\alpha_n}) \, d \mu_{n,1} + \int_{\Omega} (1- e^{-\beta_n}) \, d \mu_{n,2} \\ &=   \int_{\Omega} (1- e^{-\alpha_0}) \, d \mu_{1} + \int_{\Omega} (1- e^{-\beta_0}) \, d \mu_{2} \leq T_c( \mu_1,\mu_2).
\end{align*}
Therefore, the limit of the costs is the cost of the limits and the inequalities are equalities. We deduce that every weak limit of  $\{ \gamma_n\}$ is an optimal plan, and also that $\alpha_0$ and $\beta_0$ are the unique maximizers for the dual problem of $\mu_1$ and $\mu_2$, proving the claim. 
\end{proof}

\begin{proposition}\label{prop:increasing} Let us consider a increasing sequence of lower semi-continuous costs $c_n(x,y)$ and let us denote by $c (x,y)= \lim_n c_n(x,y)$. Then for every $\mu_1, \mu_2 \in \mathcal{M}_+(\R^d)$ we have $ T_{c_n}( \mu_1, \mu_2) \uparrow T_c(\mu_1,\mu_2)$. Moreover 
\begin{itemize}
\item[(i)] any weak limit of optimal plans $\gamma_n$ for $T_{c_n}(\mu_1, \mu_2)$ is optimal for $T_c(\mu_1, \mu_2)$;
\item[(ii)] if $\phi_n, \psi_n$ are optimal potentials for $T_{c_n}(\mu_1,\mu_2)$, we have $\phi_n \to \phi$  in $L^1(\mu_1)$ and $\psi_n \to \psi$ in $L^1(\mu_2)$, where $\phi$ and $\psi$ are optimal potentials for $T_c$;
\item[(iii)] in the case $c=c_l$ and $c_n=f_n(|x-y|)$  (as in Lemma \ref{lem:f_n}) we have also that $(\phi_n,\nabla \phi_n) \to (\phi,\nabla \phi)$ in $L^2(\mu_1)$ and similarly for $\psi_n$. 
\end{itemize}

\end{proposition}

\begin{proof} As in the previous proof we take $\gamma_n$ as minimizers for the primal problem of $T_{c_n} (\mu_1, \mu_2)$. They form a pre compact set and so up to subsequences they converge to $\gamma$, which is \emph{a priori} a suboptimal plan for $T_c(\mu_1, \mu_2)$. Let us fix $m >0 $ and then we know that for any $n \geq m$ we have $c_n \geq c_m$ and so
$$T_{c_n} (\mu_1, \mu_2) \geq \Ent{ (\gamma_n )_1 }{\mu_1} +  \Ent{ (\gamma_n )_2}{\mu_2} +  \int c_m \, d \gamma_n.$$
Now, using the semicontinuity of the entropy and the semicontinuity of $c_m$ we get
$$ \liminf_{n \to \infty} T_{c_n} (\mu_1, \mu_2) \geq  \Ent{ (\gamma)_1 }{\mu_1} +  \Ent{ (\gamma)_2}{\mu_2} +  \int c_m \, d \gamma.$$
Taking the supremum in $m$ and then the definition of $T_c$ we get
\begin{equation}\label{eqn:Tcn}
\liminf_{n \to \infty} T_{c_n}  (\mu_1, \mu_2) \geq  \Ent{ (\gamma)_1 }{\mu_1} +  \Ent{ (\gamma)_2}{\mu_2} +  \int c \, d \gamma \geq T_{c}  (\mu_1, \mu_2).
\end{equation}
Noticing that $T_c(\mu_1, \mu_2) \geq T_{c_n} (\mu_1, \mu_2)$ we can conclude. In particular $\gamma$ is optimal since we have equality in all inequalities of \eqref{eqn:Tcn}.

In order to prove (ii) notice that, since in every inequality we had equality, in particular we have $\Ent{ (\gamma_n )_1 }{\mu_1} \to \Ent{ (\gamma)_1 }{\mu_1}$. Since $\gamma_n \weakto \gamma$ and $(\gamma_n)_1 = (1-\phi_n)\mu_1$ and $\gamma_1 =(1-\phi)\mu_1$ we conclude  by Lemma \ref{lem:conv1}.

In the case we are in the hypotheses of (iii), we have that 
$$\gamma_n= (\id, h_n(\nabla \phi) )_{\#} (1-\phi_n)\mu_1 \weakto (\id, h (\nabla \phi ))_{\#} (1-\phi) \mu_1,$$
where $h_n$ is converges pointwise to $h$. Then by  Lemma \ref{lem:conv2} we deduce that $\nabla \phi_n \to \nabla \phi$ in measure with respect to $\mu_1$. Using Proposition \ref{prop:stime} we have 
$$ \int 4\phi_n^2 +  | \nabla \phi_n|^2 \, d \mu_1 \leq 4T_{c_n}(\mu_1, \mu_2) \leq 4T_c(\mu_1, \mu_2) = \int 4\phi^2+ |\nabla \phi|^2 \, d \mu_1,$$
where the last inequality can be proven as the first part of Proposition \ref{prop:stime} in the case $c=\clog$. Now, let $v_n = (2\phi_n , \nabla \phi_n)$ and $v=(2\phi , \nabla \phi)$. Since $\limsup_{n} \int |v_n|^2 \, d \mu_1 \leq \int |v|^2 \, d \mu_1$, we have $v_n \weakto w$ in $L^2(\mu_1)$,  up to subsequences: however since $v_n \to v$ in measure  we conclude $v_n \weakto v$ in $L^2(\mu_1)$ but then using $\| v\|_2^2 \geq \lim \| v_n\|_2^2$ we finally conclude $v_n \to v $ in $L^2(\mu_1)$.
\end{proof}

The following estimate allows to capture the infinitesimal behaviour of the entropy-transport metrics.

\begin{proposition}\label{prop:stime} Let $\mu_1 \in \mathcal{M}_+(\R^d)$ be an absolutely continuous measure and let $\phi:=1-e^{-\alpha}$ where $\alpha$ is the potential relative to $\mu_1$ in the minimization problem $T_{c_n}( \mu_1, \mu_2)$. Then we have that
 $$ \int_{\R^d} (| \nabla \phi|^2 + 4\phi^2) \, d \mu_1 \leq 4T_{c_n}(\mu_1, \mu_2); $$
 moreover for every $f \in C^2_c ( \R^d)$ we have
$$ \left| \int_{\R^d} f \, d (\mu_2 - \mu_1) + \frac 12 \int_{\R^d} \nabla f \cdot \nabla \phi \, d \mu_1 + 2 \int_{\R^d} f \phi \, d \mu_1 \right| \leq 5 \| f \|_{C^2} T_{c_n}( \mu_1, \mu_2), $$
where $\| f\|_{C_2} = \| f\|_{\infty } + \| \nabla f \|_{\infty} + \| D^2 f \|_{\infty}$.
\end{proposition}

\begin{proof} In the sequel we will work always $\gamma$-a.e., where $\gamma$ is the optimal plan for $T_{c_n}( \mu_1, \mu_2)$. We have $\alpha(x) + \beta (y) =c(x,y)= f_n (|y-x|)$ and $|\nabla \alpha(x)| = f_n' (|y-x|)$. We first find an upper bound for the gradient term, using an inequality from Lemma \ref{lem:f_n}:
\begin{align*}
 \int_{\R^d}  | \nabla \phi|^2 \, d \mu_1& = \int_{\R^d} | \nabla \alpha|^2 e^{-2 \alpha } \, d \mu_1 = \int_{\R^{2d}} | f_n'(|y-x|)|^2 e^{- \alpha  } \, d \gamma  \\
 & \leq \int_{\R^{2d}} 4(e^{c} -1 ) e^{- \alpha} \, d \gamma = \int_{\R^{2d}} 4e^{\alpha+ \beta} e^{-\alpha} \, d \gamma - 4\int_{\R^{2d}} e^{- \alpha} \, d \gamma \\
 &= 4\int e^{\beta} \, d \gamma - 4 \int e^{-\alpha} \, d \gamma = 4 \,\mu_2(\R^d) - 4 \int e^{- 2 \alpha} \, d \mu_1\, .
\end{align*}
Adding the term $4\int_{\R^d} | \phi|^2 \, d \mu_1$ allows to prove the first inequality:
\begin{align*}
\int_{\R^d} ( | \nabla \phi|^2 + 4| \phi |^2) \, d \mu_1 & \leq 4 \mu_2(\R^d) +  4 \int \left[(1- e^{- \alpha} )^2 - e^{-2\alpha} \right] d \mu_1 \\
&= 4 \left( \mu_2(\R^d) + \mu_1(\R^d) - 2 (e^{-\alpha} \mu_1)(\R^d) \right) \\
&=4T_{c_n}(\mu_1, \mu_2)\, .
\end{align*}
As a byproduct, we have also shown: 
\begin{equation}\label{eqn:exp_est}
\int_{\R^{2d}} (e^{c} -1 ) e^{- \alpha} \, d \gamma  \leq T_{c_n}(\mu_1, \mu_2).
\end{equation}
For the second part we will split the estimate in several parts: %
$$ 
\int_{\R^d} f \, d \mu_2 - \int_{\R^d} f \, d \mu_1 + \frac 12 \int_{\R^d} \nabla f \cdot\nabla \phi \, d \mu_1 + 2\int_{\R^d} f \phi \, d \mu_1= 
$$
$$
= \int_{\R^{2d}} \left( f(y) e^{\beta}  - f(x) e^{\alpha} + \frac 12 \nabla f(x) \cdot  \nabla \alpha  + 2f(x) (1- e^{-\alpha}) e^{\alpha} \right)\, d \gamma =
$$
\begin{align*}
\int_{\R^{2d}} \Bigl(f(y) - f(x) -&\nabla f(x) \cdot (y-x) \Bigr) e^{- \alpha} \, d\gamma \\
&+ \int_{\R^{2d}} \nabla f(x) \cdot \left( \frac {\nabla \alpha(x)}{2} + (y-x)e^{-\alpha} \right) \, d \gamma \\ 
&+ \int_{\R^{2d}} f(y)( e^{\beta}-e^{-\alpha}) \, d \gamma + \int_{\R^{2d}} f(x) (1-e^{-\alpha})^2 e^{\alpha}\, d \gamma \\
&= (I) + (II) + (III) + ( IV)
\end{align*}

Now for (I) we use the Lagrange formula for the remainder in Taylor expansion and then, using  \eqref{eqn:exp_est} and Lemma \ref{lem:f_n} (ii), namely $e^c-1 \geq |x-y|^2$, we get 
\begin{align*}
(I) 
&\leq \frac 12 \| D^2f \|_{\infty} \int_{\R^{2d}} |y-x|^2 e^{-\alpha} \, d \gamma \\& \leq \frac 12\|D^2f\|_{\infty} \int_{\R^{2d}} (e^c-1)e^{-\alpha} \, d \gamma \\
& \leq \frac 12 \| D^2 f\|_{\infty} \, T_{c_n} (\mu_1, \mu_2).
\end{align*}
For the second term: 
\begin{align*}
(II) &  = \int_{\R^{2d}}  \nabla f \cdot \left( \frac { \nabla \alpha (x)}{2} + (y-x) \right) e^{-\alpha} d \gamma + \frac 12 \int_{\R^{2d}}  \nabla f \cdot  \nabla \alpha (x)  \cdot (1- e^{-\alpha}) \, d \gamma \\
& \leq  \| \nabla f \|_{\infty} \left( \int_{\R^{2d}} \left| \frac { \nabla \alpha (x)}{2} + (y-x) \right| e^{-\alpha} \, d \gamma  +  \frac 12 \int_{\R^{2d}} | \nabla \alpha | \cdot |1-e^{-\alpha}|\, d \gamma \right)  \\
&=  (IIa) + (IIb) 
\end{align*}
Now we have $\nabla \alpha (x) = \frac{(x-y)}{|x-y|} f_n' (|x-y|)$ and in particular we have 
$$ 2\left| \frac { \nabla \alpha (x)}{2} + (y-x) \right| = | f_n' (z) - 2z|,$$
where $z=|x-y|$ and we can verify that $|f_n'(z)-2z| \leq 4(e^{f_n(z)}-1)$ independently of $n$. In fact if $f_n'(z) \geq 1$ or $z \geq 1$ this is obvious since we have
$$ | f_n'(z) -2z| \leq \max \{ f_n'(z), 2z\} \leq \max \{ f_n'(z), 2z\}^2 \leq 4(e^{f_n}-1),$$
where in the last inequality we used Lemma \ref{lem:f_n} (ii). In the case $z \leq 1$ instead we have that $f_n'(z)=f'(z)= 2 \tan (z)$ (by Lemma \ref{lem:f_n} (i)), and so, calling $t=\tan (z)$, and using that $|\arctan''(t)| \leq 1$ for every $t$ we have
$$ | f_n'(z) -2z| =2 | \tan (z) -z| = 2 |t- \arctan(t)| \leq t^2 =\frac 14  f_n'(z)^2 \leq e^{f_n(z)}-1.$$
In particular we obtain
\begin{align*}
 (IIa) &= \frac{\|\nabla f\|_{\infty}}2 \int_{\R^{2d}} | f_n'(z) -2z|\, e^{-\alpha} d \gamma \\ 
 &\leq\| \nabla f\|_{\infty} \int_{\R^{2d}} 4(e^c-1) \, e^{-\alpha}  d \gamma \leq 4 \| \nabla f\|_{\infty}\cdot  T_{c_n}(\mu_1, \mu_2).
\end{align*}
Then we use the inequality $ab \leq \frac {a^2}{4c} + c b^2 $ with $a=| \nabla \alpha|$, $b=|1-e^{-\alpha}|$ and $c=e^{\alpha}$ in order to get
\begin{align*}
 (IIb)  &= \frac { \| \nabla f\|_{\infty}}2 \int_{\R^d} | \nabla \alpha| \cdot |1-e^{-\alpha}| e^{-\alpha} \, d \mu_1 \\
 & \leq  \frac { \| \nabla f\|_{\infty}}2 \int_{\R^d} \big( \frac 14 | \nabla \alpha|^2 e^{-2\alpha} + (1-e^{-\alpha})^2 \big) \, d \mu_1 \\ & =\frac {\| \nabla f\|_{\infty}}8 \int_{\R^d} |\nabla \phi|^2 + 4 \phi^2 \, d \mu_1 \leq  \frac { \| \nabla f\|_{\infty}}2 T_{c_n}( \mu_1, \mu_2).
 \end{align*}
Then we have
\begin{align*}
 (III) &\leq \int_{\R^{2d}} |f| (e^{\beta+\alpha} - 1)e^{-\alpha} \, d \gamma \\
 &\leq \| f \|_{\infty} \int_{\R^{2d}} (e^c-1) e^{-\alpha} \, d \gamma \leq \| f\|_{\infty} T_{c_n}(\mu_1, \mu_2)
 \end{align*}
and in the end we conclude with 
$$ (IV) \leq \int_{\R^{d}} |f||1-e^{-\alpha}|^2 \, d \mu_1 \leq \| f \|_{\infty} \int_{\R^d} | \phi|^2 \, d \mu_1 \leq \| f\|_{\infty} T_{c_n}(\mu_1, \mu_2). \qedhere
$$
\end{proof}

\subsection{One step of the scheme}
From now on, we work on a bounded domain $\Omega \subset \R^d$. For a measure $\mu \in \mathcal{M}_+(\Omega)$ which is absolutely continuous and of density bounded by $1$, and a cost function $c$, consider the problem
\begin{equation}\label{eq:onestep}
\prox^{c}_{\tau G} (\mu) \eqdef \argmin \left \{  - \lambda \int_{\Omega} \rho + \frac { T_c (\mu, \rho ) }{2 \tau } \; : \; \rho \in \mathcal{M}_+ (\Omega) , \, \rho \leq 1 \right\}
\end{equation}
which corresponds to as implicit Euler steps as introduced in \eqref{eq:metricEulerScheme}: notice however that for a general cost $c$, the optimal entropy-transport cost $T_c$ is not always the square of a distance. 
We first show that this \emph{proximal operator} is well defined.

\begin{proposition}[Existence and uniqueness]\label{prop:existence step}
If $2\tau\lambda<1$ and $c:\Omega^2 \to [0,\infty]$ is a strictly convex, proper, lower semicontinuous, increasing function of the distance, then $\prox^{c}_{\tau G}$ is a well defined map on $\{\mu \in L^1_+(\Omega)\; : \; \mu\leq 1\}$, that is, the minimization problem admits a unique minimizer. Denoting $\rho=\prox^{c}_{\tau G}(\mu)$, it holds
\begin{itemize}
\item[(i)] $\sqrt{\rho(\Omega)} \leq \frac { 1+ 2 \lambda \tau }{1-2 \lambda \tau } \sqrt{\mu(\Omega)}$;
\item[(ii)] $T_c (\mu, \rho) \leq 2 \tau \lambda (\rho(\Omega)  - \mu(\Omega)) \leq \frac{(4\lambda \tau)^2}{(1-2\lambda \tau)^2} \mu(\Omega)$.
\end{itemize}
\end{proposition}

\begin{proof}
The definition of the proximal operator requires to solve a problem of the form
\[
\min_{\gamma\in \mathcal{M}_+(\Omega \times \Omega)} 
I(\gamma) \quad \text{where}\quad I(\gamma) \eqdef \int_{\Omega\times \Omega} c d \gamma + \mathcal{F}_1(\gamma_1) + \mathcal{F}_2(\gamma_2) 
\]
where $\mathcal{F}_1(\gamma_1) = \Ent{\gamma_1}{\mu}$ and $\mathcal{F}_2(\gamma_2) = \inf_{\rho\leq 1}\{ \Ent{\gamma_2}{\rho} -2\tau\lambda \rho(\Omega)\}$ are both convex functions of the marginals of $\gamma$ (note that the optimal $\rho$ in the definition of $\mathcal{F}_2$ is explicit using the pointwise first order optimality conditions : $\rho(x) = \min\{ 1, \gamma_2(x)/(1-2\tau\lambda)\}$, for a.e. $x\in \Omega$). In order to prove the existence of a minimizer, we give a proof that does not assume compactness of $\Omega$ since we need the mass estimates anyways. 

Remark that $\gamma=\mu\otimes \mu$ is feasible and that $I$ is weakly lower semicontinuous so we only have to show that the closed sublevel set $S = \{\gamma\in \mathcal{M}_+(\Omega^2)\; : \; I(\gamma) \leq I(\mu \otimes \mu)\}$ is tight, and thus compact, in order to prove the existence of a minimizer. 
Let us consider $\gamma \in S$ and $\rho(x)= \min\{ 1, \gamma_2(x)/(1-2\tau\lambda)\}$: then we have 
\[
 - \lambda \mu(\Omega) \geq I(\mu \otimes \mu) \geq I(\gamma) \geq  - \lambda \rho(\Omega) + \frac { T_c(\mu, \rho)}{2\tau}.
\]
Then, using  Lemma \ref{lem:mass} we obtain
\[
 - \lambda \mu(\Omega) \geq - \lambda \rho(\Omega) + \frac1{2\tau} { \left(\sqrt{\mu(\Omega)}- \sqrt{\rho(\Omega)}\right)^2};
\]
by rearranging the terms, it follows $\sqrt{\rho(\Omega)} \leq \frac { 1+ 2 \lambda \tau }{1-2 \lambda \tau } \sqrt{\mu(\Omega)}$,
so we have a bounded mass as long as $2\lambda \tau <1$. Thanks to the positivity of $\mathcal{H}$, this implies that $\mathcal{F}_2(\gamma_2)$ is lower bounded for  $\gamma\in S$ and thus both $\mathcal{F}_1(\gamma_1)$ and $\int c d \gamma$ are upper bounded (since nonnegative).
Incidentally, we obtained also (ii), an estimate for the dissipated energy 
\begin{equation}\label{eqn:lipest}
T_c (\mu, \rho) \leq 2 \tau \lambda (\rho(\Omega)  - \mu(\Omega)) \leq \frac{(4\lambda \tau)^2}{(1-2\lambda \tau)^2} \mu(\Omega).
\end{equation}
The upper bound on $\mathcal{F}_1(\gamma_1)$ and the superlinear growth at infinity of the entropy imply that $S$ is bounded and $\{ \gamma_1 \; : \; \gamma \in S\}$ is tight (see \cite[Prop. 2.10]{liero2015optimal}). Let $\epsilon>0$ and $K_1$ be a compact set such that $\gamma_1(\Omega\setminus K_1)<\epsilon/2$ for all $\gamma\in S$. The assumptions on $c$ guarantee that the set $K_\lambda := \{(x,y)\in K_1\times \Omega\; : \; c(x,y)\leq\lambda\}$ is compact for $\lambda\in \R$, and by the Markov inequality, $\int_{K_1\times \Omega} c d \gamma \geq \lambda \gamma((K_1\times \Omega) \setminus K_\lambda)$. Consequently, for $\lambda$ big enough, it holds for all $\gamma \in S$:
\[
\gamma(\Omega^2\setminus K_\lambda) = \gamma_1(\Omega\setminus K_1) + \gamma((K_1\times \Omega) \setminus K_\lambda) \leq \epsilon/2 + \epsilon/2 \leq \epsilon
\]
which proves the tightness of $S$ and shows the existence of a minimizer.




For uniqueness, observe that if $\gamma$ is a minimizer, then it is a deterministic coupling.  Indeed, $\gamma$ is an optimal coupling for the cost $c$ between its marginals, which are absolutely continuous. But $c$ satisfies the \emph{twist condition} which garantees that any optimal plan is actually a map, because $c$ is a strictly convex function of the distance. 

Now take two minimizers $\gamma^a$ and $\gamma^b$ and define $\tilde\gamma =\frac12\gamma^a + \frac12\gamma^b$ which is also a minimizer, by convexity. Note that $\tilde\gamma$ must be a deterministic coupling too, which is possible only if the maps associated to  $\gamma^a$ and $\gamma^b$ agree almost everywhere on $(\spt(\gamma^a_1)\cap \spt(\gamma^b_1)) \times \Omega$. Finally, since all the terms in the functional are convex, it most hold $\mathcal{F}_1(\tilde{\gamma}_1) = \frac12 \mathcal{F}_1(\gamma^a_1) +\frac12 \mathcal{F}_1(\gamma^b_1)$. But $\mathcal{F}_1$ is strictly convex so $\gamma^a_1=\gamma^b_1$ and thus $\gamma^a = \gamma^b$. This, of course, implies the uniqueness of $\rho$ which is explicitly determined from the optimal $\gamma$.
\end{proof}
 
%
We now use the dual formulation  in order to get information on the minimizer.

\begin{proposition}\label{prop:h1} Let us consider $\rho = \prox^{c}_{\tau G} ( \mu )$. Then there exists a Lipschitz optimal potential  $\phi$  relative to $\rho$ for the problem $T_c( \rho, \mu)$ such that $\phi \leq 2\tau \lambda$ and
$$ \rho (x) =  \begin{cases} 1 \quad & \text{ if } \phi < 2\tau \lambda, \\ [0,1] & \text{ if  } \phi =2\tau \lambda . \end{cases} $$
\end{proposition}

\begin{proof} In the problem \eqref{eq:onestep}, let us consider a competitor $\bar{\rho} \leq 1$ and define $\rho_{\ep} = \rho + \ep ( \bar{ \rho} - \rho )$. Since $\rho_{\ep}$ is still admissible as a competitor we have that
$$- \lambda \int_{\Omega} \rho + \frac { T_c (\mu, \rho ) }{2 \tau } \leq - \lambda \int_{\Omega} \rho_{\ep} + \frac { T_c (\mu, \rho_{\ep} ) }{2 \tau }.$$
We can now use the fact that, if  $\phi_{\ep}$ and $\psi_{\ep}$ are the maximizing potentials in the dual formulation for $T_c(\mu, \rho_{\ep})$, we have
$$T_{c}(\mu, \rho_{\ep}) = \int \phi_{\ep} \, d \rho_{\ep} + \int \psi_{\ep} \, d \mu \;   \qquad \text{and}\qquad T_c(\mu, \rho) \geq \int \phi_{\ep} \, d \rho + \int \psi_{\ep} \, d \mu,$$
because $\phi_{\ep}, \psi_{\ep}$ are admissible potentials also for $\mu$ and $\rho$. In particular we deduce that
$$ - \lambda \int_{\Omega} \rho + \frac 1 {2\tau}  \int \phi_{\ep} \, d \rho \leq - \lambda \int_{\Omega} \rho_{\ep} + \frac 1 {2\tau} \int  \phi_{\ep} \, d \rho_{\ep};$$
$$ 0 \leq - \lambda \int_{\Omega} \ep (\bar{\rho} - \rho) + \ep \frac 1 {2\tau} \int  \phi_{\ep} \, d (\bar{\rho}-\rho). $$
Dividing this inequality by $\ep$ and then let $\ep \to 0$, using that $\phi_{\ep} \to \phi_0$ locally uniformly by Proposition \ref{prop:uniform}, we get
$$ \int_{\Omega}  (\phi_0 - 2\lambda \tau ) \, d \rho \leq \int_{\Omega} (\phi_0- 2 \lambda \tau ) \, d \bar{ \rho}  \qquad \forall 0\leq \bar{\rho} \leq 1,$$
where $\phi_0$ is an optimal (Lipschitz) potential relative to $\rho$. This readily implies 
$$ \rho (x) =  \begin{cases} 1 \quad & \text{ if } \phi_0 < 2\tau \lambda \\ [0,1] & \text{ if  } \phi_0 =2\tau \lambda \\ 0 & \text{ if } \phi_0 > 2\tau \lambda. \end{cases}. $$
Now it is sufficient to take $\phi = \inf\{ 2 \tau \lambda, \phi_0\}$ and we have that $\phi$ is still an admissible potential since $(1-\phi)(1-\psi) \geq (1- \phi_0)(1-\psi) \geq e^{-c}$ and moreover we have $\int \phi \, d \rho = \int \phi_0 \, d \rho$ and so $\phi$ also is optimal.
\end{proof}

\begin{lemma}[Stability of $\prox$]\label{lem:stability} Let $(c_n)_{n\in\mathbb{N}}$ be an increasing sequence of Lipschitz cost functions, each satisfying the hypotheses of Proposition \ref{prop:existence step} and let $c$ be the limit cost. Then $\rho_n = \prox^{c_n}_{\tau G}(\mu)$ converges weakly to $\rho=\prox^{c}_{\tau G}(\mu)$.
\end{lemma}

\begin{proof} By Proposition \ref{prop:existence step} we know that $\{\rho_n\}$ have equi-bounded mass and in particular, up to subsequences,  $\rho_n \weakto \bar{\rho}$ which, in particular, will be supported on $\Omega$ and it is such that $\bar{\rho} \leq 1$. Fix $m \in \mathbb{N}$ and $n\geq m$; by the minimality of $\rho_n$ we know that for every $\nu$ we have
\begin{align*}  
\frac { T_{c_m} (\rho_n, \mu) }{2\tau} - \lambda \int_{\Omega} \rho_n 
& \leq  \frac { T_{c_n} (\rho_n, \mu) }{2\tau} - \lambda \int_{\Omega} \rho_n  \\
& \leq  \frac { T_{c_n} (\nu, \mu) }{2\tau} - \lambda \int_{\Omega} \nu  
\leq  \frac { T_c (\nu, \mu) }{2\tau} - \lambda \int_{\Omega} \nu.
\end{align*}
Taking now the limit as $n \to \infty$, using the continuity of $T_{c_m}$ (see Proposition \ref{prop:uniform}), we get
$$  - \lambda \int_{\Omega} \bar{\rho} + \frac { T_{c_m} (\bar{\rho}, \mu) }{2\tau}  \leq   - \lambda \int_{\Omega} \nu + \frac { T_c (\nu, \mu) }{2\tau}.$$
Now we can take the limit $m \to \infty$ and use that $T_{c_m} \uparrow T_c$ (see Proposition \ref{prop:increasing}) in order to get
$$  - \lambda \int_{\Omega} \bar{\rho} + \frac { T_{c} (\bar{\rho}, \mu) }{2\tau}  \leq   - \lambda \int_{\Omega} \nu + \frac { T_c (\nu, \mu) }{2\tau},$$
that is, $\bar{\rho}$ is a minimizer for the limit problem and so by the uniqueness $\bar{\rho}=\rho$.
\end{proof}

\begin{lemma}\label{lem:1step} Let us consider $\rho = \prox^{\clog}_{\tau G}(\mu)$. If $\Omega$ is a regular domain\footnote{we need that $\Omega$ is an $H^1$ extension domain, that is, there exists $C>0$ such that for every $f \in H^1(\Omega)$ there exists $\tilde{f}$ with $\tilde{f}|_{\Omega}=f$ and $\| \tilde{f}\|_{H^1( \mathbb{R}^d)} \leq C \| f \|_{H^1(\Omega)}$.} then there exists $p \in H^1(\Omega)$ that verifies $p \geq 0 $, $p(1-\rho)=0$, such that
$$ \int_{\Omega} (| \nabla p|^2 + 4(p-\lambda)^2 )\, d \rho \leq \frac{\dist^2(\rho, \mu)}{\tau^2}$$
and such that for all $f\in C^2(\Omega)$,
$$ \left|\int_{\Omega}  f \, d(\mu - \rho)  -  \tau \int_{\Omega} ( \nabla p \cdot \nabla f  + 4(p- \lambda) f ) \, d\rho \right| \leq  C\| f \|_{C^2}  \dist^2 (\rho, \mu). $$
\end{lemma}

\begin{proof} We first use the approximated problem $\rho_n = \prox^{c_n}_{\tau G} (\mu)$. By Lemma \ref{lem:stability} we know that $\rho_n \weakto \rho$.
%
Using Proposition \ref{prop:h1} we know there exists optimal potentials  $\phi_n$ such that, calling $p_n = \frac 1{2\tau} ( 2 \tau \lambda - \phi_n )$, we have $p_n \in H^1(\Omega)$, $p_n\geq 0$ and $p_n ( 1-\rho_n)=0$. Moreover, thanks to Proposition \ref{prop:stime} we have also that 
\begin{equation}\label{eq:stiman} 
\tau^2 \int_{\Omega} \left(|\nabla p_n|^2 + 4(\lambda -p_n)^2\right) dx =  \frac 14 \int \left( |\nabla \phi_n|^2 + 4|\phi_n|^2 \right) d \rho_n \leq T_{c_n} ( \rho_n, \mu)
\end{equation}
\begin{equation}\label{eq:contn} 
\left| \int_{\Omega} f \, d \mu  - \int_{\Omega} f \, d \rho_n + \tau \int_{\Omega}( -\nabla p_n \cdot \nabla f  + 4(\lambda -p_n)f)\, dx \right| \leq 5 \| f \|_{C^2} T_{c_n}( \rho_n, \mu)
\end{equation}
In particular, using Equations \eqref{eqn:lipest} and \eqref{eq:stiman}, we get that $p_n$ is equibounded in $H^1(\Omega)$. Thanks to the hypothesis on $\Omega$, there exist a sequence $\tilde{p}_n$ equibounded in $H^1(\mathbb{R}^d)$ such that $\tilde{p}_n|_{\Omega}=p_n$; in particular there is a subsequence of $\tilde{p}_n$ that is weakly converging in $H^1(\mathbb{R}^d)$ and strongly in $L^2$ to some $p \in H^1$, $p\geq 0$. Since we have  $\rho_n \weakto \rho$ in duality with $C_b$ and so also in duality with $L^1$, thanks to the $L^{\infty}$ bound, we get that $\int_{\Omega} p_n(1-\rho_n) \to \int_{\Omega} p (1-\rho)$ and so we have $\int_{\Omega} p (1-\rho) \, dx =0$ that implies $p(1-\rho)=0$ almost everywhere in $\Omega$, since $\rho \leq 1 $ and $p\geq 0$. Now we can pass to the limit both Equation \eqref{eq:stiman} and \eqref{eq:contn} getting the conclusion.
\end{proof}


\subsection{Convergence of minimizing movement}

We consider an initial density $\rho_0\in L^1_+(\Omega)$ and define the discrete gradient flow scheme as introduced in \eqref{eq:metricEulerScheme} which depends an a \emph{time step} $\tau>0$
\begin{equation}\label{eq:JKO discr}
\begin{cases} 
\rho_0^{\tau} = \rho_0 \in L^1_+(\Omega) \\
 \rho^\tau_{n+1} = \prox^{\clog}_{\tau G} ( \rho_n ) \quad \text{ for }n\geq1\, , 
\end{cases}
\end{equation}
define $p^{\tau}_{n+1}$ as the pressure relative to the couple $\rho^{\tau}_n, \rho^{\tau}_{n+1}$ (provided by Lemma \ref{lem:1step}) and extend all these quantities in a piecewise constant fashion as in Definition \ref{def:gradientflow} in order to obtain a family of time dependant curves $(\rho^\tau,p^\tau)$: 
\begin{equation}\label{eq:JKO interp}
\begin{cases} 
\rho^{\tau}(t) =\rho^{\tau}_{n+1} \\
p^{\tau}(t) =p^{\tau}_{n+1}\
\end{cases} 
 \text{ for } t \in ] \tau n , \tau (n+1)] 
\end{equation}
The next lemmas exhibit the regularity in time of $\rho^\tau$, which improves as $\tau$ diminishes. 
\begin{lemma}\label{lem:bounded squares}
There exists a constant $C>0$ such that for any $\tau>0$, the sequence of minimizers satisfy
\[
\sum_{k} \dist^2(\rho^\tau_k,\rho^\tau_{k+1}) \leq \tau C\, .
\]
\end{lemma}
\begin{proof}
By optimality, $\rho_{k+1}^\tau$ satisfies
$
\dist^2(\rho^\tau_{k},\rho^\tau_{k+1}) \leq  2\tau\left( G(\rho^\tau_{k})-G(\rho^\tau_{k+1}) \right) .
$
By summing over $k$, one obtains a telescopic sum which is upper bounded by $2\tau (G(\rho_0)-\inf G)$ and $\inf G = -\lambda|\Omega|$ is finite because $\Omega$ has a finite Lebesgue measure.
\end{proof}

The consequence of this bound is a H\"older property, a standard result for gradient flows.
\begin{lemma}[Discrete H\"older property]\label{lem:discrete holder}
Let $\rho_0\leq1$ and $T>0$. There exists a constant $C>0$ such that for all $\tau>0$ and $s,t\geq 0$, it holds
\[
\dist(\rho^{\tau}_t , \rho^{\tau}_s ) \leq C(\tau+ |t-s|)^{1/2}\, .
\]
In particular, if $(\tau_n)_{n\in\mathbb{N}}$ converges to $0$, then, up to a subsequence, $\rho^{\tau_n}$ weakly converges to a $\frac12$-H\"older curve $\rho$. 
\end{lemma}
\begin{proof} The first result is direct if $s$ and $t$ are in the same time interval $]\tau k,\tau(k+1)]$ so we suppose that $s-t \geq \tau$ and let $k,l$ be such that $t\in ]\tau (k-1) , \tau k]$ and $s \in ]\tau (l-1), \tau l]$.
By the triangle inequality and the Cauchy-Schwarz inequality, one has
\[
\dist(\rho^{\tau}(t),\rho^{\tau}(s)) \leq \sum_{i=k}^{l-1} \dist(\rho^{\tau}_i,\rho^{\tau}_{i+1}) 
\leq \left( \sum_{i=k}^{l-1} \dist(\rho^{\tau}_i,\rho^{\tau}_{i+1})^2\right)^{\frac12} \left( l-k\right)^{\frac12}\, .
\]
By using Lemma \ref{lem:bounded squares} and the fact that $|l-k|\leq 1+ |s-t|/\tau$, the first claim follows.

As for the second claim, let us adapt the proof of Arzel\`a-Ascoli theorem to this discontinuous setting. Let $(q_i)_{i\in\mathbb{N}}$ be a enumeration of $\mathbb{Q} \cap [0,T]$ and let $\tau_k\to 0$. Since $\{\rho^{\tau_k}(q_1)\}$ is bounded in $L^{\infty}\cap L^1(\Omega)$, it is weakly pre-compact and thus there is a subsequence $\tau_{k^{(1)}}$ such that $\rho^{\tau_{k^{(1)}}}(q_1)$ converges. By induction, for any $i\in \mathbb{N}$, one can extract a subsequence $\tau_{k^{(i)}}$ from $\tau_{k^{(i-1)}}$ so that $\rho^{\tau_{k^{(i)}}}(q_i)$ converges. 

Now, we form the diagonal subsequence $(\rho^m)_{m\in \mathbb{N}}$ whose $m$-th term is the $m$-th term in the $m$-th subsequence $\rho^{\tau_{k^{(m)}}}$. By construction, for every rational $q_i \in [0,T]$, $\rho^m(q_i)$ converges. Moreover, for every $t\in [0,T]$ and for $q \in [0,T]\cap \mathbb{Q}$
\[
\dist(\rho_n(t),\rho_m(t)) \leq C(\tau_{k^{(n)}} +|t-q|)^{\frac12} +  \dist(\rho_n(q),\rho_m(q)) + C(\tau_{k^{(m)}} +|t-q|)^{\frac12}\,
\]
by the triangle inequality and the discrete H\"older property. So by taking $q$ close enough to $t$, one sees that $(\rho_m(t))_{m\in \mathbb{N}}$ is Cauchy and thus converges. Let us denote $\rho$ the limit. For $0\leq s \leq t \leq T$, and $m\in \mathbb{N}$, it holds
\[
\dist(\rho(t),\rho(s)) \leq \dist(\rho(t),\rho_m(t)) + \dist(\rho_m(t),\rho_m(s)) + \dist(\rho_m(s),\rho(s)) \, .
\]
In the right-hand side, the middle term is upper bounded by $C(\tau_{k^{(m)}}+|s-t|)^\frac12 \to C|s-t|^{\frac12} $ and the other terms tend to $0$ so by taking the limit $m\to \infty$, one obtains the $\frac12$-H\"older property.
\end{proof}

Collecting all the estimates established so far, we obtain an existence result.
\begin{proposition}[Existence of solutions]\label{prop:gradientflowexistence}
The family $(\rho^\tau,p^\tau)$ defined  in \eqref{eq:JKO interp} admits weak cluster points $(\rho,p)$ as $\tau\downarrow 0$ which are solutions to the evolution PDE \eqref{eq:mainPDE} on ${[0,T]}$, for all $T>0$.
\end{proposition}
\begin{proof}
Define the sequence of momentums $E^{\tau}_{n} = -\rho^\tau_{n} \nabla p^\tau_{n}$ and source terms $D^\tau_n= 4\rho^\tau_{n} ( \lambda - p^\tau_n)$ and extend these quantities in a piecewise constant fashion as in \eqref{eq:JKO interp}. Gathering the results, let us first show that there exists a constant  $C=C(T, \rho_0)$ such that:
\begin{itemize}
\item[(i)] $\rho_t^{\tau} p_t^{\tau} = p_t^{\tau}$ and  $p_t^{\tau} \geq 0$;
\item[(ii)] $\int_0^T \int_\Omega (|\nabla p^\tau|^2 + |p^\tau|^2)dx dt \leq \int_0^T \int_{\Omega} (| E^{\tau}_t |^2 + |D^{\tau}_t|^2 ) \, dx \, dt \leq C$;
\item[(iii)] $\left| \int_{\Omega} \phi\, d  \rho^{\tau}_t  - \int_{\Omega} \phi \, d \rho^{\tau}_s  - \int_s^t \int_{\Omega}( E^{\tau}_r \cdot \nabla \phi + D^{\tau}_r \phi ) \,dx \, dr\right| \leq C\Vert \phi \Vert_{C^2} \max\{\tau,\sqrt{\tau}\}$, for all $\phi\in C^2(\Omega)$;
\item[(iv)] $\int_0^T\int_\Omega |\nabla p^\tau| dx dt < C$.
\end{itemize}

Property (i) is a direct from Lemma \ref{lem:1step} and the definition of the curves $p^\tau$ and $\rho^\tau$. One then proves (ii) and (iv) by using Lemma \ref{lem:1step} and property (i). Indeed, one has 
\[
\int_{\Omega} (|E^{\tau}_n|^2 +|D^\tau_n|^2)dx  \leq \int_{\Omega} (|\nabla p_n^\tau|^2+4(\lambda-p_n)^2)d\rho_n\leq \frac1{\tau^2} \dist^2(\rho^\tau_{n-1},\rho^\tau_{n})
\]
Integrating now the interpolated quantities it follows, by Lemma \ref{lem:bounded squares},
\[
\int_0^T\int_\Omega (|E^{\tau}_t|^2 +|D^\tau_t|^2)dxdt 
\leq \sum_{n=1}^{\lceil \frac T\tau \rceil} \tau \int_{\Omega} (|E^{\tau}_n|^2 +|D^\tau_n|^2)dx
\leq \frac1{\tau} \sum_n \dist^2(\rho^\tau_{n-1},\rho^\tau_{n}) \leq C\, .
\]
Property (iii) is obtained from Lemma \ref{lem:1step} in a similar way. Indeed, for all $\phi\in C^2(\Omega)$, by denoting $I_a^b = \int_a^b \int_\Omega (E^\tau_r \nabla \phi +D^\tau_r \phi) dx dr$, 
\begin{align*}
\left\vert \int_\Omega \phi (\rho^\tau_t - \rho^\tau_s )dx - I_s^t  \right\vert &=
 \Big\vert  I_s^{k\tau} - I_{t}^{l\tau}+ \\
 &  \sum_{i=k}^{l-1} \left[ \int_\Omega \phi (\rho^\tau_{i+1} - \rho^\tau_i)dx - \tau \int_{\Omega} (E^\tau_{i+1} \nabla \phi +D^\tau_{i+1} \phi)dx \right] \Big\vert \\
& \leq |I_s^{k\tau}| +|I_{kl}^t| + C\Vert \phi \Vert_{C^2} \sum_{i=k}^{l-1} \dist^2(\rho_i^\tau,\rho_{i+1}^\tau)
\end{align*}
where $k=\lceil \frac s\tau\rceil$ and $l=\lceil \frac t\tau\rceil$. By Lemma \ref{lem:bounded squares}, the last term is bounded by $C\tau$ and by Lemma \ref{lem:1step}, $|I_s^{k\tau}|$ and $|I_{t}^{l\tau}|$ are controlled by $C\sqrt{\tau}$ :
$$
|I_s^{k\tau} | \leq \tau |\int_{\Omega} (E_k^\tau\nabla \phi + D_k^\tau \phi)dx | \leq \Vert \phi \Vert_{H^1(\Omega)} \dist (\rho_{k-1},\rho_{k}) \leq C\Vert \phi \Vert_{H^1(\Omega)} \sqrt{\tau}.
$$
So property (iii) is shown.

Let us now take a sequence $\tau_k\to 0$ and pass those relations to the limit. Recall that from the discrete H\"older property (Lemma \ref{lem:discrete holder}), up to a subsequence, $(\rho^{\tau_k})$ admits a weakly continuous limit $(\rho_t)_{t\in [0,T]}$. Moreover, thanks to the $L^2$-norm bound (ii) we have, up to a subsequence $( E^{\tau_k} , D^{\tau_k} ) \rightharpoonup ( E, D)$. In particular, looking at relation (iii), we obtain, for all $\phi \in C^2(\Omega)$,
\[
\int_\Omega \phi d \rho_t -\int_\Omega \phi d \rho_s = \int_s^t \int_\Omega (E_r \cdot \nabla \phi + D_r \phi)dx dr\, .
\]
which means that $(\rho,E,D)$ is a weak solution of $ \partial_t \rho_t + \nabla \cdot  E_t  = D_t$.

In order to conclude it remains to prove that $D= 4(\lambda - p) \rho$ and $E=-\rho\nabla p$ for some admissible pressure field $p$.  As $(p^\tau)$ is a bounded family  in the Hilbert space $L^2([0,1],H^1(\Omega))$, there exist weak limits $p$ when $\tau\to 0$. The property $p\geq0$ is obvious but the Hele-Shaw complementary relation $p(1-\rho)=0$ is more subtle. We obtain it by combining the spatial regularity of $p^\tau$ with the time regularity of $\rho^\tau$ as was done for the Wasserstein case in \cite{maury2010macroscopic}. Using the complementary relation $p^\tau(1-\rho^\tau)=0$ one has for all $0<a<b<T$:
\[
0 = \frac1{b-a} \int_a^b \int_{\Omega} p^\tau_t(x)(1-\rho^\tau_a(x))dx dt + \frac1{b-a} \int_a^b \int_{\Omega} p^\tau_t(x)(\rho^\tau_a(x)-\rho^\tau_t(x))dx dt \, .
\]
Denoting $p_{[a,b]} := \int_a^b p_t dt$, the first term converges to $\int_\Omega p_{[a,b]}(x)(1-\rho_a(x))dx$ because $p^\tau_{[a,b]}$ converges to $p_{[a,b]}$ --- weakly in $H^1(\Omega)$ and thus strongly in $L_{loc}^2(\Omega)$ since $\Omega$ bounded --- and $\rho^\tau_a$ converges weakly to $\rho_a$ in duality with $L^1(\Omega)$. Additionally, for every Lebesgue point $a$ of $t \mapsto p_t$ (seen as a map in the separable Hilbert space $L^2(\Omega)$) we have
\[
\int_\Omega p_{[a,b]}(x)(1-\rho_a(x))dx dt \xrightarrow[b\to a]{} \int_{\Omega} p_a(x)(1-\rho_a(x))dx\, .
\]
For the second term, we use Lemma \ref{lem:H-1} (stated below) and obtain
\begin{align*}
\int_a^b \int_{\Omega} p^\tau_t(x)(\rho^\tau_a(x)-\rho^\tau_t(x))dx dt 
&\leq 2 \int_a^b \Vert  p^\tau_t\Vert_{H^1(\R^d)} \dist(\rho^\tau_a,\rho^\tau_t)dt \\
& \leq C\sqrt{\tau+(b-a)} \left(\int_a^b  \Vert p^\tau_t \Vert^2_{H^1(\R^d)} dt \right)^\frac12 \left(\int_a^b dt \right)^\frac12\\
& \leq C(b-a)\sqrt{1+\tau/(b-a)} \left(\int_a^b \Vert  p^\tau_t \Vert^2_{H^1(\R^d)} dt \right)^\frac12\, .
\end{align*}
Notice that since the geodesics used in Lemma \ref{lem:H-1} may exit the domain $\Omega$ we have to use the $H^1$ norm of $p^{\tau(t, \cdot)}$ on the whole $\R^d$, in the sense that we extend it, and thanks to the regularity of $\Omega$ we have $\| p^{\tau}(t, \cdot)\|_{H^1(\R^d)} \leq C \| p^{\tau}(t, \cdot) \|_{H^1(\Omega)}$. In this way the functions $t\mapsto \Vert  p^\tau(t,\cdot)\Vert^2_{H^1(\R^d)}$ are $\tau$-uniformly bounded in $L^1([0,1])$ and so admit a weak cluster point $\sigma \in \mathcal{M}_+([0,T])$ as $\tau\to 0$. Thus, for a.e. $a\in[0,T]$,
$$
\lim_{\tau \to 0} \frac1{b-a} \int_a^b \int_\Omega p^\tau_t(x)(\rho^\tau_a(x)-\rho^\tau_t(x))dx dt \leq C \sqrt{\sigma([a,b])} \xrightarrow[b\to a]{} 0\, .
$$
As a consequence, for a.e.\ $a$, $\int_\Omega p_a(x)(1-\rho_a(x)) dx = 0$, and since $p\geq0$ and $\rho \leq 1$, this implies $p(1-\rho)=0$ a.e. 

We are finally ready to recover the expressions for $E$ and $D$ by writing this quantities as linear functions of $p$ and $\rho$ which are preserved under weak convergence and then plugging the nonlinearities back using $p(1-\rho)=0$. For $D^\tau \rightharpoonup D$ on has 
\[
D^\tau = 4(\lambda-p^\tau)\rho^\tau = 4(\lambda \rho^\tau - p^\tau) \underset{\tau\to 0}\rightharpoonup 4(\lambda \rho - p) = 4(\lambda - p)\rho = D,
\]
while for $E^\tau\rightharpoonup E$ one has
\[
E^\tau = -\rho^\tau \nabla p^\tau = -\nabla p^\tau \underset{\tau\to 0}\rightharpoonup -\nabla p = -\rho \nabla p = E\,.\qedhere
\]
\end{proof}

In the proof, we used the following Lemma which is well-known for the case of Wasserstein distances, and illustrates a link between $\dist$ and $H^{-1}$ norms. Its proof is a simple adaptation of the Wasserstein case, given the geodesic convexity result from \cite{liero2017convexity}. Notice that, as in the Wasserstein case, this Lemma can be generalized to the case where $L^p$ bounds on the measures imply a comparison between $\dist$ and the $W^{-1,q}$ norm, where $\frac 1p + \frac 2q=1$. 
\begin{lemma}\label{lem:H-1}
Let $(\mu,\nu)\in \mathcal{M}_+(\R^d)$ be absolutely continuous measures with density bounded by a constant $C$. Then, for all $\phi \in H^1(\R^d)$, it holds
\[
\int_{\R^d} \phi d (\mu-\nu) \leq 2\sqrt{C} \Vert \phi \Vert_{H^1(\R^d)} \dist(\mu,\nu)\, .
\]
\end{lemma}
\begin{proof}
Consider a minimizing geodesic $(\rho_t)_{t\in [0,1]}$ between $\mu$ and $\nu$ for the distance $\dist$ and $(v,\alpha)\in L^2([0,1],L^2(\rho_t))$ the associated velocity and growth fields. These quantities are the optimal variables in \eqref{eq:WFRgeodesic} and they satisfy the constant speed property $\Vert v_t \Vert^2_{L^2(\rho_t)} +\Vert \alpha_t \Vert^2/4 = \dist^2(\mu,\nu)$ for a.e.\ $t\in [0,1]$ (see \cite{chizat2015interpolating,kondratyev2015new,liero2015geodesics}). Moreover, by Theorem \ref{thm:convexity}, $L^{\infty}$ bounds are preserved along geodesics. Let us take $\phi \in H^1(\R^d)$ and notice that by approximation we can suppose that its support is bounded; then it holds
\begin{align*}
\int_{\R^d} \phi d(\mu-\nu) &= \int_0^1 \frac{d}{dt} \left( \int_{\R^d} \phi \rho_t\right) dt = \int_0^1\int_{\R^d} \left( \nabla \phi \cdot v_t + \phi \alpha_t \right) \rho_t dx dt \\
&\leq \sqrt{\int_0^1 \int_{\R^d} (|\nabla \phi|^2 + 4 |\phi|^2)\rho_t dx dt} \sqrt{\int_0^1 \int_{\R^d} (|v_t|^2 + \frac14 |\alpha_t|^2)\rho_t dx dt} \\
& \leq 2\sqrt{C}\Vert \phi \Vert_{H^1(\R^d)} \dist(\mu,\nu)\, .\qedhere
\end{align*}
\end{proof}

This lemma relies on an announced result of geodesic convexity for $\dist$ \cite{liero2017convexity}. We also rely on this result for proving uniqueness.
 
  \begin{theorem}\label{thm:convexity} Let us consider $\mu_t$ be a geodesic of absolutely continuous measures connecting the two absolutely continuous measures $\mu_0$ and $\mu_1$. Then, for every $m>1$ we have that $t \mapsto \int \bigl( \frac{d\mu_t}{d\mathcal{L}^d } \bigr) ^m \, dx$ is convex. In particular if $\mu_1, \mu_0 \leq C \mathcal{L}^d $, we have $\mu_t \leq C \mathcal{L}^d$ too.
 \end{theorem}
 

\subsection{Proof of uniqueness}
\begin{proposition}[Uniqueness]\label{prop:uniqueness}
If $\Omega$ is convex, every solution of the PDE \eqref{eq:mainPDE} is an $\mathrm{EVI}_{(-2\lambda)}$ solution of gradient flow of $G$ in the metric space $(\mathcal{M}_+(\Omega),\dist)$ and we have uniqueness for \eqref{eq:mainPDE}
\end{proposition}
\begin{proof}
We follow the same lines as \cite{di2016uniqueness}, using the convexity result from Theorem \ref{thm:convexity}. Let us consider two solutions $(\rho_t^1,p_t^1)$ and $(\rho_t^2, p_t^2)$. Let us assume we can prove that we have (distributionally)

\begin{equation}\label{eqn:derivative} \frac{ d}{dt} \dist^2 (\rho_t^1, \rho_t^2) = \int_{\Omega} -\nabla \phi_t \cdot \nabla p_t^1 + 4\phi_t ( \lambda - p_t^1) \, d \rho_t^1 +\int_{\Omega}  -\nabla \psi_t \cdot \nabla p_t^2 + 4\psi_t ( \lambda - p_t^2) \, d \rho_t^2, 
\end{equation}
where $\phi_t, \psi_t$ is a couple of optimal potentials for $\rho_t^1, \rho_t^2$. Then using Lemma \ref{lem:appendix} we conclude
$$ \frac{ d}{dt} \dist^2 (\rho_t^1, \rho_t^2)  \leq 4\lambda \int \phi_t \, d \rho_t^1 + 4\lambda  \int \psi_t \, d \rho_t^2 = 4\lambda \dist^2 (\rho_t^1, \rho_t^2),$$
and so by Gr\"onwall's lemma it follows $\dist^2(\rho_t^1, \rho_t^2) \leq e^{2\lambda t} \dist^2(\rho_t^1, \rho_t^2)$. So we are left to prove \eqref{eqn:derivative} in the distributional sense. Notice \eqref{eqn:derivative} is true if we can prove that for every $0 < s < r <T$ we have
$$\dist^2(\rho_r^1, \rho_r^2) - \dist^2(\rho_s^1, \rho_s^2) = \int_s^r \left( \int_{\Omega} \nabla \phi_t \cdot v_t^1 + \phi_t r_t^1 \, d \rho_t^1 +\int_{\Omega}  \nabla \psi_t \cdot v_t^2 + \psi_t  r_t^2 \, d \rho_t^2 \right) \, dt,$$
where we can suppose $\partial_t \rho_t^1 + \nabla \cdot ( v_t^1 \rho_t^1)= r_t^1 \rho_t^1$  with $\iint (|v_t^1|^2 + (r_t^1)^2)\, d \rho_t^1 \, d t< \infty$ and similarly for $\rho_t^2$. Let us fix $n$ and consider $T_{c_n}(\rho_t^1, \rho_t^2)$ and a couple of optimal potentials $\phi_n, \psi_n$. In particular for every $s$ we have
$$ T_{c_n}(\rho_s^1 , \rho_s^2) \geq \int \phi_n \, d \rho_s^2 + \int \psi_n \rho_s^2,$$
with equality for $s=t$. Now, with a slight modification of \cite[Lemma 2.3]{di2016uniqueness} we can prove that there exists a full measure set where we can differentiate both sides an the derivatives are equal. In particular, using that $t\mapsto T_{c_n}(\rho_t^1, \rho_t^2)$ is absolutely continuous, we get 
\begin{multline*}
T_{c_n}(\rho_r^1, \rho_r^2) - T_{c_n}(\rho_s^1, \rho_s^2) =\\ \int_s^r \left( \int_{\Omega} \nabla \phi_{n,t} \cdot v_t^1 + \phi_{n,t} r_t^1 \, d \rho_t^1 +\int_{\Omega}  \nabla \psi_{n,t} \cdot v_t^2 + \psi_{n,t}  r_t^2 \, d \rho_t^2 \right) \, dt,
\end{multline*}
and then letting $n \to \infty$ we conclude, using that $(\phi_{n,t}, \nabla \phi_{n,t} ) \to (\phi_t, \nabla \phi_t)$ in $L^2(\rho_t^1)$ thanks to Proposition~\ref{prop:increasing} (iii).

The EVI Characterization is easily deduced from those previous computations. Taking a solution $(\rho_t,p_t)_{t\in [0,T]}$ and any $\mu \in \mathcal{M}_+(\Omega)$ such that $\mu \leq 1$, we have, by denoting $(\phi_t,\psi_t)$ the optimal potentials for $(\rho_t,\mu)$,
\[
\frac12 \frac{d}{dt} \dist^2(\rho_t,\mu) = \frac 12 \int_\Omega (-\nabla \phi_t \cdot \nabla p_t + 4\phi_t ( \lambda - p_t))\, d\rho_t \leq 2\lambda \int_\Omega \phi_t \, d\rho_t
\]
by Lemma \ref{lem:appendix} and we conclude using Theorem \ref{th:duality} (i) and (iv), which proves that (denoting with $\gamma_t$ the optimal plan between $\rho_t$ and $\mu$) one has 
\begin{align*}
2\lambda \int_\Omega \phi_t \, d\rho_t &= 2\lambda \rho_t(\R^d) - 2 \lambda \gamma_t(\R^d \times \R^d)  \\
&  = 2\lambda \rho_t(\R^d) + \lambda \left( \dist^2(\rho_t,\mu) -  \rho_t(\R^d) -  \mu( \R^d)\right) \\
& \leq  G(\mu) -G(\rho_t) + \lambda \dist^2(\rho_t,\mu). \qedhere
\end{align*}
\end{proof}

\section{Numerical scheme}\label{sec:numericalscheme}
The characterization of the tumor growth model \eqref{eq:mainPDE} as a gradient flow suggests a constructive method for computing solutions through the time discretized scheme \eqref{eq:metricEulerScheme}. In this section we describe a consistent spatial discretization, an optimization algorithm and numerical experiments.

First, let us recall that the resolution of one step of the scheme involves, for a given time step $\tau>0$ and previous step $\mu\in L^1_{+}(\Omega)$, such that $\mu\leq \mathcal{L}^d$ to compute
\begin{equation}\label{eq:timestepbis}
\nu \in \argmin \left\{ 2\tau G(\nu)  + \dist(\nu,\mu)^2 \right\}
\end{equation}
According to Proposition \ref{prop:existence step}, by using the optimal entropy-transport problem \eqref{eq:KFdef} and exchanging the infima, this problem can be written in terms of one variable $\gamma$ which stands for the unknown coupling
\begin{equation}\label{eq:onesstepcoupling}
\min_{\gamma\in \mathcal{M}_{+}(\Omega^2)} \left\{  
\int_{\Omega^2} c(x,y)d\gamma + \Ent{\gamma_1}{ \mu} + \inf_{\nu\in \mathcal{M}_{+}(\Omega)} \left\{ \Ent{\gamma_2}{\nu} + 2\tau G(\nu) \right\}
\right\}\, ,
\end{equation}
which admits a unique minimizer $\gamma^*$ and the optimal $\nu^{*}$ can be recovered from $\gamma^{*}$  through the first order pointwise optimality conditions as
\[
\nu^{*}=\min\{ 1, \gamma_2^{*}/(1-2\tau\lambda)\}.
\]
The subject adressed in this Section is thus the numerical resolution of \eqref{eq:onesstepcoupling}.

\subsection{Spatial discretization}
Let $\mathcal{W}=(W_{i}, x_i)_{i=1}^N$ be a pointed partition of a compact domain $\Omega \subset \R^d$ where $x_{i}$ is a point which belongs to the interior of the set $W_{i}$ for all $i$ (in our experiments, $W_i$ will always be a $d$-dimensionnal cube and $x_i$ its center). We denote by $\diam \mathcal{W}$ the quantity $\max_i \diam W_{i}$. An atomic approximation of $\mu\in \mathcal{M}_+(\Omega)$ is given by the measure
 \[
\mu_\mathcal{W} = \sum_{i=1}^N \alpha_{i} w_i \delta_{x_{i}}
 \]
where $w_i\eqdef \mathcal{L}^d(W_i)$ is the (positive) Lebesgue measure of $W_i$, $\alpha_{i}\eqdef \mu(W_{i})/w_i$ are the locally averaged densities and $\delta_{x_{i}}$ is the Dirac measure of mass $1$ concentrated at the point $x_{i}$. This is a proper approximation since for a sequence of partitions $(\mathcal{W}_k)_{k\in\mathbb{N}}$ such that $\diam \mathcal{W}_k \to 0$ then $\mu_{\mathcal{W}_k}$ converges weakly to $\mu$ (indeed, $\mu_{\mathcal{W}_k}$ is the pushforward of $\mu$ by the map $W_{k,i} \ni x\mapsto x_{k,i}$ which converges uniformly to the identity map as $k\to \infty$).

Now assume that we are given a vector $\alpha\in \R_+^N$.
For a discrete coupling $\gamma \in \R^{N\times N}$ seen as a square matrix, let $J$ be the convex functional defined as
\begin{equation}\label{eq:functionalJ}
J(\gamma) \eqdef \langle c, \gamma \rangle  + F_{1}(\gamma w) + F_{2}(\gamma^{T} w)
\end{equation}
where $\gamma w$, $\gamma^T w$ are matrix/vector products, $ \langle c, \gamma \rangle \eqdef \sum_{i,j} c(x_{i},y_{j}) \gamma_{i,j}w_iw_j$ and
\begin{align*}
F_1: \R^N \ni \beta &\mapsto H(\beta|\alpha)\\
F_2: \R^N \ni \beta &\mapsto \min_{s\in [0,1]^N} \left\{ H(\beta|s) - 2\lambda\tau \sum_i s_iw_i \right\}
\end{align*}
and, for $\alpha,\beta\in \R_+^N$, the discrete relative entropy is $H(\beta|\alpha) \eqdef  \sum_i H(\beta_i|\alpha_i)$ where
\begin{equation*}
H(\beta_i|\alpha_i) \eqdef \begin{cases}
 (\beta_i\log(\beta_i/\alpha_i)-\beta_i + \alpha_i) w_i &\text{if $\beta_i\geq0$ and $\alpha_i>0$}\\
0 & \text{if  $\beta_i=0$ and $\alpha_i=0$}\\
+\infty & \text{otherwise.}
\end{cases}
\end{equation*}
With these definitions, solving the finite dimensional convex optimization problem
\begin{equation}\label{eq:discreteproblem}
\gamma^{*} \in \argmin_{\gamma\in \R_{+}^{N\times N}} J(\gamma)
\end{equation}
is nothing but solving a discrete approximation of \eqref{eq:onesstepcoupling} where the maximum density constraint is not with respect to the Lebesgue measure anymore, but with respect to its discretized version. This is formalized in the following simple proposition. 

\begin{proposition}\label{prop:minimizerdiscret}
Let $\mathcal{W}$ be a partition of $\Omega$ as above and let $\gamma^*$ be obtained through \eqref{eq:discreteproblem}. Then the measure $\nu \eqdef \sum_i \beta_i w_i \delta_{w_i}$ where $\beta = \min \{ 1, ((\gamma^*)^T w)/(1-2\lambda\tau)\}$ does not depend on the choice of $\gamma^*$ and is a minimizer of
\begin{equation*}
\inf_{\nu\in \mathcal{M}_+(\Omega)} \dist^2(\mu_{\mathcal{W}},\nu) -2\tau \nu(\Omega) +\iota_C(\nu)
\end{equation*}
where $\iota_C$ is the convex indicator of the set $C$ of measures which are upper bounded by the discretized Lebesgue measure $\sum_i w_i\delta_{x_i}$.
\end{proposition}
\begin{proof}
This result essentially follows by construction. Let us denote by (P) the minimization problem in the proposition: (P) can be written as a minimization problem over couplings $\gamma\in \mathcal{M}_+(\Omega\times \Omega)$ as in \eqref{eq:onesstepcoupling}. But in this case, any feasible $\gamma$ is discrete because both marginals must be discrete in order to have finite relative entropies. Thus (P) reduces to the finite dimensional problem $\eqref{eq:discreteproblem}$ and the expression for $\beta$ is obtained by first order conditions. Finally, (P) is strictly convex as a function of $\nu$, hence the uniqueness.
\end{proof}

The following proposition guarantees that the discrete measure $\nu_k$ built in Proposition~\ref{prop:minimizerdiscret} properly approximates the continuous solution.

\begin{proposition}[Consistency of discretization]\label{prop:convergencespatialdiscret}
Let $(\mathcal{W}_k)_{k\in \mathbb{N}}$ be a sequence of partitions of $\Omega$ such that $\diam \mathcal{W}_k \to 0$ and for all $k$ compute $\nu_k$ as in Proposition \ref{prop:minimizerdiscret}. Then the sequence $(\nu_k)$ converges weakly to the continuous minimizer of \eqref{eq:timestepbis}.
\end{proposition}


\begin{proof}
As a sequence of bounded measures on a compact domain, $(\nu_k)$ admits weak cluster points. Let $\bar{\nu}$ be one of them. The fact that for all $k$, $\nu_k$ is upper bounded by the discretized Lebesgue measure $\sum_i w_i \delta_{x_i}$ implies that $\bar{\nu}$ is upper bounded by the Lebesgue measure in $\mathbb{R}^d$, since the discretized Lebesgue measure weakly converges to the Lebesgue measure.
Now, let $\sigma \in \mathcal{M}_+(\Omega)$ be any measure of density bounded by $1$. By Proposition \ref{prop:minimizerdiscret}, one has for all $k \in \mathbb{N}$,
\begin{equation*}
\dist^2(\mu_{\mathcal{W}_k},\nu_k) - 2\tau \nu_{k}(\Omega) \leq \dist^2(\mu_{\mathcal{W}_k},\sigma_{\mathcal{W}_k}) - 2\tau \sigma_{\mathcal{W}_k}(\Omega)\, .
\end{equation*}
Since the distance $\dist$ and the total mass are continuous functions under weak convergence one obtains, in the limit $k\to \infty$,
\begin{equation*}
\dist^2(\mu,\bar{\nu}) - 2\tau \bar{\nu}(\Omega) \leq \dist^2(\mu,\sigma) - 2\tau \sigma(\Omega)
\end{equation*}
which proves that $\bar{\nu}$ minimizes \eqref{eq:timestepbis}. By Proposition~\ref{prop:existence step}, this minimizer is unique.
\end{proof}

\subsection{Entropic regularization and scaling algorithm}
The discrete optimization problem \eqref{eq:discreteproblem} is a smooth finite dimensional convex optimization problem with linear constraints which could be solved with classical tools. However, the dimension of the variable $\gamma$ is typically very big ($100^{2d}$ for uniformly discretized cube $[0,1]^d$ with grid spacing $0.01$). Since for our problem it is acceptable to solve \eqref{eq:discreteproblem} with an error which is negligible compared to the (time and space) discretization error, so we suggest to use more efficient methods based on entropic regularization.

Cuturi has shown in~\cite{CuturiSinkhorn} that, for solving the discrete optimal transport problem, adding the entropy of the coupling to the Kantorovich optimal transport functional, leads to a simple, parallelizable and linearly convergent algorithm for solving each step, known as Sinkhorn's algorithm. This algorithm has then been subsequently generalized \cite{benamou2015iterative}, applied to Wasserstein gradient flows \cite{peyre2015entropic}, and extended to unbalanced optimal transport problems \cite{chizat2016scaling}. The framework of the latter includes the functional \eqref{eq:discreteproblem}.  In \cite{chizat2016scaling,schmitzer2016stabilized}, it has been described how to take care of stability issues caused by small regularization parameter while preserving the nice structure of the algorithm, which makes it possible to solve \eqref{eq:discreteproblem} with high precision in a reasonable time.

\subsubsection{Algorithm and convergence}
%
The  method of entropic regularization consists in minimizing, instead of \eqref{eq:discreteproblem}, the strictly convex problem
\begin{equation}\label{eq:entropicreg}
\min_{\gamma\in \R^{N \times N}} J(\gamma) + \epsilon H(\gamma)
\end{equation}
where $J$ is defined in \eqref{eq:functionalJ}, $\epsilon>0$ is a small regularization parameter and $H$ is, as above, the relative entropy with respect to the discrete Lebesgue measure
$$
H(\gamma)\eqdef \begin{cases}
\sum_{i,j} (\gamma_{i,j}\log(\gamma_{i,j})-\gamma_{i,j}+1)w_iw_j &\text{if $\gamma_{i,j}\geq0$, for all $i,j$} \\
+\infty & \text{otherwise,}
\end{cases}
$$
with the convention $0\log 0 = 0$. Of course, one recovers the unregularized problem as $\epsilon\to 0$, as stated in the next Proposition whose proof is simple (see~\cite{chizat2016scaling}).
\begin{proposition}\label{prop:convergenceepsilon}
Denoting by $\gamma_\epsilon$ and $\gamma^*$ minimizers of \eqref{eq:entropicreg} and \eqref{eq:discreteproblem} respectively, one has 
\[
J(\gamma_\epsilon)-J(\gamma^*) \leq \epsilon (H(\gamma^*)-H(\gamma_\epsilon)) = o(\epsilon) 
\quad \text{ and } \quad \gamma_\epsilon \to \gamma^*.
\]
\end{proposition}
\begin{remark}
For classical optimal transport, precise convergence results of the minimizers are known in the continuous setting~\cite{leonard2012schrodinger}. The convergence of entropy regularized JKO schemes is also shown in~\cite{2017-carlier-SIMA} if the inequality $-\epsilon \log(\epsilon)\leq C\tau^2$ is preserved when taking the joint limit $\epsilon,\tau \to 0$ (on a continuous spatial domain).
\end{remark}
The algorithm we suggest to minimize \eqref{eq:entropicreg}, referred to as \emph{Iterative scaling algorithm} in~\cite{chizat2016scaling}, is then simple to write. It consists in letting $b^{(0)}=[1, \dots, 1]^T\in \R^N$ and iteratively computing
\begin{equation}\label{eq:scalingiter}
a^{(\ell+1)} = 
\frac{\prox^H_{F_1/\epsilon}(K(b^{(\ell)}\odot w))}{K(b^{(\ell)}\odot w)}
, \qquad 
b^{(\ell+1)} = 
\frac{ \prox^H_{F_2/\epsilon}(K^T(a^{(\ell+1)}\odot w)) }{ K^T(a^{(\ell+1)}\odot w) }
\end{equation}
where $K$ is the matrix $K \eqdef (e^{c(x_i,x_j)/\epsilon})_{i,j=1,\dots,N}$, the division is performed elementwise with the convention $0/0=0$, $\odot$ denotes elementwise multiplication and the proximal operator of a function $F$ with respect to the relative entropy is defined as
\begin{align*}
\prox^H_F &:  \R_+^N \ni a \mapsto \argmin_{b\in \R^N_+} \{ F(b) + H(b|a) \}.
\end{align*}
In our precise case, these iterates have the following explicit form
\begin{equation*}
a^{(\ell+1)} = \left(\frac{\alpha}{K(b^{(\ell)}\odot w))}\right)^{\tfrac{1}{1+\epsilon}}
\end{equation*}
and 
\begin{equation*}
b^{(\ell+1)} = \begin{cases}
(1-2\tau\lambda)^{\tfrac{-1}{\epsilon}} & \text{if $K^T(a^{(\ell+1)}\odot w) \leq (1-2\tau\lambda)^{\tfrac{1+\epsilon}{\epsilon}}$} \\
(K^T(a^{(\ell+1)}\odot w))^{\tfrac{-1}{1+\epsilon}} & \text{otherwise.}
\end{cases}
\end{equation*}
where exponentiation and comparison are performed elementwise and we recall that $\alpha\in \R^n$ is the vector describing the discretization of $\mu$. This algorithm can be interpreted as an alternate maximization in the dual variables. We sketch a proof of this fact, see~\cite{chizat2016scaling} for more details.
\begin{proposition}
The iterative scaling algorithm corresponds to alternate maximization on the dual problem.
\end{proposition}
\begin{proof}[Sketch of proof]
The dual problem to \eqref{eq:entropicreg} reads (up to a constant)
\[
\max_{u \in \R^N, v \in \R^N} L(u,v) \eqdef \Big\{ -F_1^*(-u) - F_2^*(-v) - \epsilon \sum_{i,j} e^{\frac1\epsilon (u_i + v_j - c(x_i,y_j))}w_iw_j \Big\}
\]
where $F_1^*$ and $F_2^*$ are the convex conjugates of $F_1$ and $F_2$. Iterations  \eqref{eq:scalingiter} are obtained by performing alternate maximization on $u$ and $v$ successively. Indeed, for $v^{(\ell)}$ fixed, the partial maximization problem and its dual read
\[
\max_{u\in \R^N} -F_1^*(-u) -\epsilon \sum_i w_i e^{u_i/\epsilon} (K(b^{(\ell)}\odot w))_i
\quad \text{and} \quad
\min_{s\in \R^N}  F_1(s) + \epsilon H(s|K(b^{(\ell)}\odot w))
\]
respectively, with the primal dual relationship $s = (e^{u_i/\epsilon})_{i=1}^N$ at optimality.
Since the term coupling $(u,v)$ in the dual functional is smooth, it is known that $L(u^{(\ell)}, v^{(\ell)})$ converges to the maximum of $L$ at a guaranteed rate $O(1/\ell)$.
\end{proof}

In our specific instantiation, a linear rate of convergence can be proved by looking at the explicit form of the iterates.
\begin{proposition}\label{prop:algoconvergence}
The sequence of dual variables $v^{(\ell)} \eqdef \epsilon \log b^{(\ell)}$ converges linearly in $\ell_\infty$-norm to the optimal (regularized) dual variable $v_\epsilon$, more precisely
\[
\Vert v^{(\ell)} - v_\epsilon \Vert_\infty \leq \frac{\Vert v^{(1)} - v_\epsilon  \Vert_\infty}{(1+\epsilon)^{\ell}}.
\]
and the same holds true for $u^{(\ell)}\eqdef \epsilon \log a^{(\ell)}$. Moreover, the matrix $\gamma^{(\ell)} \eqdef (a^{(\ell)}_i K_{i,j} b^{(\ell)}_j)$ converges to the minimizer $\gamma_{\epsilon}$ of  \eqref{eq:entropicreg}.
\end{proposition}
\begin{proof}
For two vectors $(x,y)$ in $\R_+^N$ sharing the same set of indices with positive entries $I\subset \{1,\dots, N\}$, the Thompson metric is defined as $d_T(x,y) \coloneqq \max_{i\in I} \vert \log x_i/y_i \vert$. It is known that order preserving (or reversing) maps of degree $r$ are $|r|$-Lipschitz for the Thompson metric~\cite[Chap. 2]{lemmens2012nonlinear}. In particular, looking at the form of the iterates, we deduce
\[
d_T(a^{(\ell+1)},a^{(\ell)}) \leq \frac1{1+\epsilon} d_T(b^{(\ell)},b^{(\ell-1)}) 
\quad \text{and} \quad
d_T(b^{(\ell+1)},b^{(\ell)}) \leq d_T(a^{(\ell+1)},a^{(\ell)}).
\]
It follows that the sequence $(b^{(\ell)})_{\ell\in \mathbb{N}}$ is Cauchy and thus converges since $(\R_+^M,d_T)$ is complete. Denoting $b_\epsilon$ the limit, it is a fixed point of the iterates and thus is of the form $b_\epsilon = e^{v_\epsilon/\epsilon}$ where $v_\epsilon$ is an optimal dual variable. The fixed point property also yields
\[
d_T(b^{(\ell+1)}, b_\epsilon) \leq \frac1{1+\epsilon}d_T(b^{(\ell)}, b_\epsilon) \leq \frac1{(1+\epsilon)^\ell}d_T(b^{(1)}, b_\epsilon)
\]
and the conclusion follows by remarking that all entries of $b^{(\ell)}$ are positive for all $\ell\in \mathbb{N}$ so $d_T(b^{(\ell)}, b_\epsilon)=\frac1\epsilon \Vert v^{(\ell)}-v_\epsilon \Vert_\infty$. The reasoning also works for the sequence $(a^{(\ell)})_{\ell \in \mathbb{N}}$ and we obtain the convergence (linear for $d_T$) of $\gamma^{(\ell)}$ which is the primal minimizer from the primal-dual relationships.
\end{proof}
Finally, given the optimal regularized coupling $\gamma_\epsilon$ one recovers the regularized discrete new step through (recall Proposition~\ref{prop:minimizerdiscret})
\[
(\nu_\epsilon)_i = \min \left\{ 1, \frac{\sum_j (\gamma_\epsilon)_{ij} w_j}{1-2\lambda\tau} \right\}.
 \]
In what follows, we take  $p_{\epsilon} \eqdef (2\tau -1 + e^{-v_\epsilon})/(2\tau)$ as the expression for the regularized pressure since in the regularized version of \eqref{eq:timestepbis}, it is the term in the subgradient of the upper bound constraint at optimality. However, we do not attempt to establish a rigorous convergence result of $p_\epsilon$ to the true pressure field.

\subsubsection{Stabilization}
While iterations \eqref{eq:scalingiter} are mathematically correct, we observe in practice that the entries of $a^{(\ell)}$ and $b^{(\ell)}$ become very big for small values of $\epsilon$ and rapidly go out of the range of the standard 64 bits floating point number representation in computers. While doing all the computations in the log-domain is not desirable since matrix/vector products would become very slow for big problems, a stabilization method which allows to maintain the efficiency of the initial algorithm is possible (we refer to \cite{chizat2016scaling} for the details). It consists in writing the variables of \eqref{eq:scalingiter} as $a^{(\ell)} = \tilde{a}^{(\ell)} \exp(\tilde u^{(\ell)}/\epsilon)$ where $\tilde{a}^{(\ell)}$ will be kept of the order of $1$ by being ``absorbed'' in $\tilde u^{(\ell)}$ from time to time during an absorption step as 
\begin{align*}
\tilde u^{(\ell+1)} &= \tilde u^{(\ell)}+\epsilon \log(\tilde{a}^{(\ell)}) \quad \text{and} \quad \tilde{a}^{(\ell+1)}=1 \\
\tilde v^{(\ell+1)} &= \tilde v^{(\ell)}+\epsilon \log(\tilde{b}^{(\ell)}) \quad \text{and} \quad \tilde{b}^{(\ell+1)}=1 \\
\tilde{K}^{(\ell+1)}_{i,j} &= \exp((\tilde u^{(\ell+1)}_i+ \tilde v^{(\ell+1)}_j-c(x_i,x_j))/\epsilon).
\end{align*}
With this double parametrization, the aborbed scaling iterations are written in Algorithm \ref{alg:mainalgo} where the function $\proxdiv$ is defined as
\begin{equation}
\proxdiv_F : (\R_+^N,\R^N, \R_+^*) \ni (a,u,\epsilon) \mapsto \prox^H_{F/\epsilon}(a\odot e^{-u/\epsilon}) \oslash a.
\end{equation}
where $\odot$ and $\oslash$ denote elementwise multiplication and division, with the convention $0/0=0$. By direct computations, one finds that this operator is explicit for the functions $F_1$ and $F_2$ (defined below \eqref{eq:functionalJ}).
\begin{proposition}\label{prop:explicitproxdiv}
One has
\begin{equation*}
\proxdiv_{F_1}(s,u,\epsilon) = (\alpha/ s)^{\tfrac{1}{1+\epsilon}} \odot e^{\tfrac{-u}{1+\epsilon}}
\end{equation*}
and 
\begin{equation*}
\proxdiv_{F_2}(s,u,\epsilon) = \begin{cases}
((1-2\tau\lambda)e^u)^{\tfrac{-1}{\epsilon}} & \text{if $s \leq (1-2\tau\lambda)^{\tfrac{1+\epsilon}{\epsilon}}e^{\tfrac{u}{\epsilon}}$} \\
(s\odot e^u)^{\tfrac{-1}{1+\epsilon}} & \text{otherwise.}
\end{cases}
\end{equation*}
\end{proposition}


\begin{remark}[Remarks on implementation]
For small values of $\epsilon$, most entries of $K$ are below machine precision at initialization. Thus one needs to first approximately solve the problem \eqref{eq:entropicreg} with higher values of $\epsilon$ in order to build a good initial guess and ``stabilize'' the values of $K$. These steps are included in Algorithm \ref{alg:mainalgo}. Typically, we start with $\epsilon=1$, stop the algorithm after a few iterations, divide $\epsilon$ by $10$, and repeat until the desired value for $\epsilon$ is reached. Then only start the ``true'' scaling iterations for which a meaningful stopping criterion has to be chosen (we use $|\log \tilde a^{(\ell+1)} - \log \tilde a^{(\ell)}|_{\infty}<10^{-6}$ in our experiments). 
\end{remark}

\begin{algorithm}
\caption{Compute one step of the flow}\label{alg:mainalgo}
\begin{enumerate}
\item input $(\alpha_i)_{i=1}^N$ the discrete density of $\mu_k$,  $(w_i)_{i=1}^N$ the discretized Lebesgue measure and $E$ an array of decreasing values for the regularization parameter $\epsilon$
\item initialize $(u,v) \gets (0_N,0_N)$
\item for $\epsilon$ in $E$:
\begin{enumerate}
\item $\tilde{b} \gets \ones_N$
\item for all $i,j$, $\tilde{K}_{ij} \gets \exp((\tilde u_i + \tilde v_j-c(x_i,x_j))/\epsilon)$ 
\item while stopping criterion not satisfied:
\begin{enumerate}
\item $\tilde{a} \gets \proxdiv_{F_1}(\tilde{K}(\tilde{b} \odot w),\tilde u,\epsilon)$
\item $\tilde{b} \gets \proxdiv_{F_2}(\tilde{K}^T(\tilde{a} \odot w),\tilde v,\epsilon)$
\item if $\max(|\log \tilde{a}|,|\log \tilde{b}|)$  is too big or if last iteration:
\begin{enumerate}
\item $(\tilde u,\tilde v) \gets (\tilde u+\epsilon \log \tilde{a}, \tilde v + \epsilon \log \tilde{b})$
\item  for all $i,j$, $\tilde{K}_{ij} \gets \exp((\tilde u_i + \tilde v_j-c(x_i,x_j))/\epsilon)$ 
\item  $\tilde{b} \gets \ones_N$
\end{enumerate}
\end{enumerate}
\end{enumerate}
\item for all $i,j$, $\gamma_{i,j} \gets \exp((\tilde u_i + \tilde v_j-c(x_i,x_j))/\epsilon)$ for all $i,j$.
\item define $\beta = (\min \{ 1, (\sum_{j} \gamma_{i,j} w_j)/(1-2\lambda\tau)\})_{i=1}^N$ the discrete density of $\mu_{k+1}$
\item  define $p = (2\tau-1+e^{-\tilde v})/(2\tau)$ the new pressure
\end{enumerate}
\end{algorithm}

\subsubsection{Some remarks on convergence of the scheme}\label{sss:discretization parameters}
Gathering results from the previous Sections, we have proved that the scheme solves the tumor growth model \eqref{eq:mainPDE} when 
\begin{itemize}
\item the number of iterations $\ell\to \infty$ (Prop.~\ref{prop:algoconvergence});
\item the entropic regularization $\epsilon\to 0$ (Prop.~\ref{prop:convergenceepsilon});
\item the spatial discretization $ \diam \mathcal{W}\to 0$ (Prop. \ref{prop:convergencespatialdiscret});
\item the time step $\tau\to 0$  (Prop. \ref{prop:gradientflowexistence}).
\end{itemize}
successively, in this order. In practice, one has to fix a value for these parameters. We did not provide explicit error bounds for all these approximations, but it is worth highlighting that a bad choice leads to.a bad output.
As already known for Wasserstein gradient flows \cite[Remark 4]{maury2015pressureless}, there is for instance a \emph{locking} effect when the discretization $ \diam \mathcal{W}$ is too coarse compared to the time step $\tau$. In this case, the cost of moving mass from one discretization cell to another is indeed big compared to the gain it results in the functional. 

Let us perform some numerical experiments for one step of the flow ( the study of the effect of $\tau$ is postponed to the next Section). We fix a time step $\tau=0.005$, a domain $\Omega = [0,1]$, an initial density $\rho_0$ which is the indicator of the segment $[0.1,0.9]$ and use a uniform spatial discretization $(W_i,x_i) = ([\frac{i-1}{N}, \frac{i}{N}[,\frac{2i-1}{2N})$ for $i\in\{1,\dots,N\}$. The scaling iterations are stopped as soon as  $|\log a^{(\ell+1)}- \log a^{(\ell)}|_\infty<10^{-6}$.

We first run a reference computation, with very fine parameters ($N=8192$ and $\epsilon = 2\times 10^{-7}$) and then compare this with what is obtained with degraded parameters, as shown on Figure \ref{fig:taufixed}. On Figure \ref{fig:taufixed}(a)-(b), the error on the radius of the new step $r=(\sum_i \beta_i)/2N$ and the $\ell_\infty$ error on the pressure (with respect to the reference computation) are displayed. Since the initial density is the indicator of a segment, the new step is expected to be also the indicator of a segment, see Section \ref{sec:sphere}. On Figure \ref{fig:taufixed}(c), we display the left frontier of the new density and observe how it is smoothed when $\epsilon$ increases (the horizontal scale is strongly zoomed).
 
\begin{figure}
\centering
\subcaptionbox{Error on radius}[.3\linewidth]{
 	\resizebox{.3\linewidth}{!}{
        \includegraphics{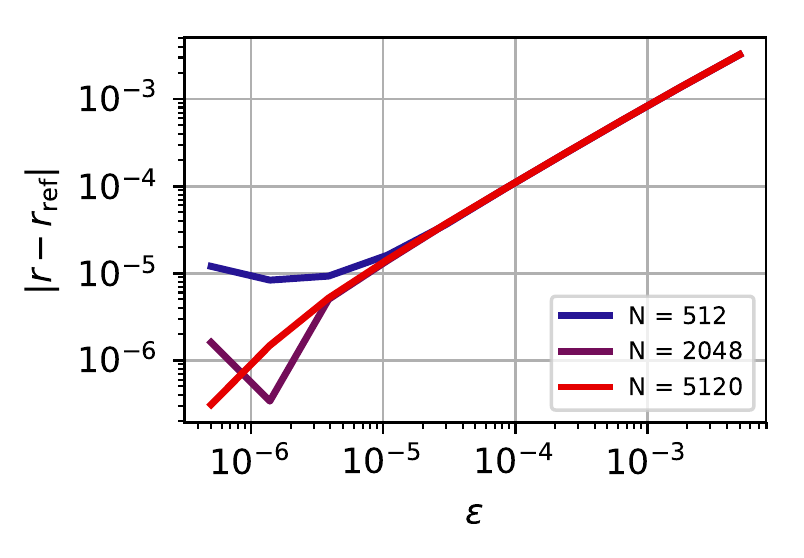}
	}}%
      \subcaptionbox{Error on pressure}[.3\linewidth]{
	\resizebox{.3\linewidth}{!}{
        \includegraphics[width=\textwidth]{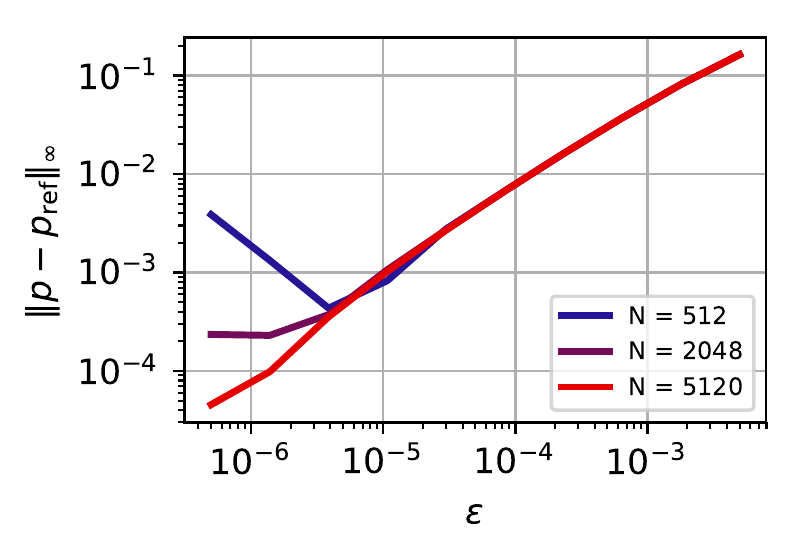}
	}}
	\subcaptionbox{Smoothing effect}[.3\linewidth]{
	\resizebox{.3\linewidth}{!}{
        \includegraphics[width=\textwidth]{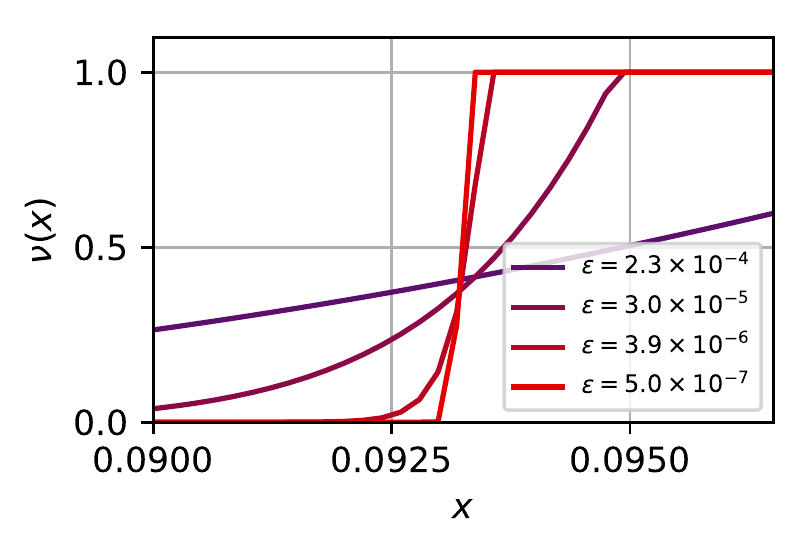}
	}}
	 \caption{Effect of discretization and entropic regularization}
	 \label{fig:taufixed}
\end{figure}

\section{Test case : spherical solutions}\label{sec:sphere}
In this section, we show that when the initial condition $\rho_0$ has unit density on a sphere and vanishes outside, then the solution of \eqref{eq:mainPDE} are explicit, using Bessel functions. Knowing this exact solutions allows us to assess the quality of the numerical algorithm.

\subsection{Explicit solution}

Let us construct the explicit solution for $\rho_0(x)= \chi_{B(0,r)}$. For $\alpha > -1$, we define the modified Bessel function of the first kind:
$$I_{\alpha} (x) \eqdef \sum_{m=0}^{\infty} \frac { 1}{ m! \Gamma ( m+ \alpha +1 )} \left ( \frac x 2 \right) ^{2m+\alpha} $$
and
$$H_{\alpha} (x) \eqdef x^{-\alpha} I_{\alpha}(x) \quad \text{and} \quad K_{\alpha}(x) \eqdef x^{\alpha} I_{\alpha}(x)\, .$$
The following mini-lemma states properties of these functions that are relevant here.
\begin{lemma}\label{lem:bessel}With the definitions as above, we have the followings:
\begin{itemize}
\item[(i)] $y = I_{\alpha}(x)$ satisfies the equation $ x^2 y''+ x y' - (x^2-\alpha^2) y =0$;
\item[(ii)] $y=H_{\alpha}(\beta x)$ satisfies the equation $ x^2 y '' + (2 \alpha +1 )x y' - \beta^2 x^2 y =0$ and, up to constants is the unique locally bounded at $0$; 
\item[(iii)] $H'_{\alpha}(x)=xH_{\alpha+1}(x)$ and $K'_{\alpha+1}(x) =x K_{\alpha} (x)$.
\end{itemize}
\end{lemma}

\begin{proof} The proof that $I_{\alpha}$ satisfies the equation is trivial and can be done coefficient by coefficient since everything is converging absolutely. Moreover it is known that all the other independent functions explode like $\log (x)$ near zero. Also the equation for $H_{\alpha}$ is easy to derive and it is clear that it is the unique regular one. As for (ii) it can be deduced straightly deriving their formulas (using $\Gamma(x+1)= x \Gamma(x)$ where required) from the definition of $I_{\alpha}$
\end{proof}

With the help of these Bessel functions, one can give the explicit solution when the initial condition is the indicator of a ball. Some properties of these solutions are displayed on Figure \ref{fig:explicitball}.
\begin{proposition}\label{prop:explicitball}
Consider the initial condition $\rho_0 = \chi_{B_{r_0}}$ for some $r_0>0$. Then there is a unique solution for \eqref{eq:mainPDE} which is the indicator of an expanding ball $\rho_t = \chi_{B_{r(t)}}$, where the radius evolves as
\[
r(t) = \frac12 K_{d/2}^{-1} \left[ e^{4\lambda t} K_{d/2}(2 r_0) \right]\, .
\]
Moreover, the pressure is radial and given by $p_t(x) = \lambda\left(1 - \frac{H_{d/2-1}(2|x|)}{H_{d/2-1}(2r(t))}\right)_+$.
\end{proposition}

\begin{figure}
\centering
\subcaptionbox{Radius vs time}[.3\linewidth]{
 	\resizebox{.3\linewidth}{!}{
        \includegraphics{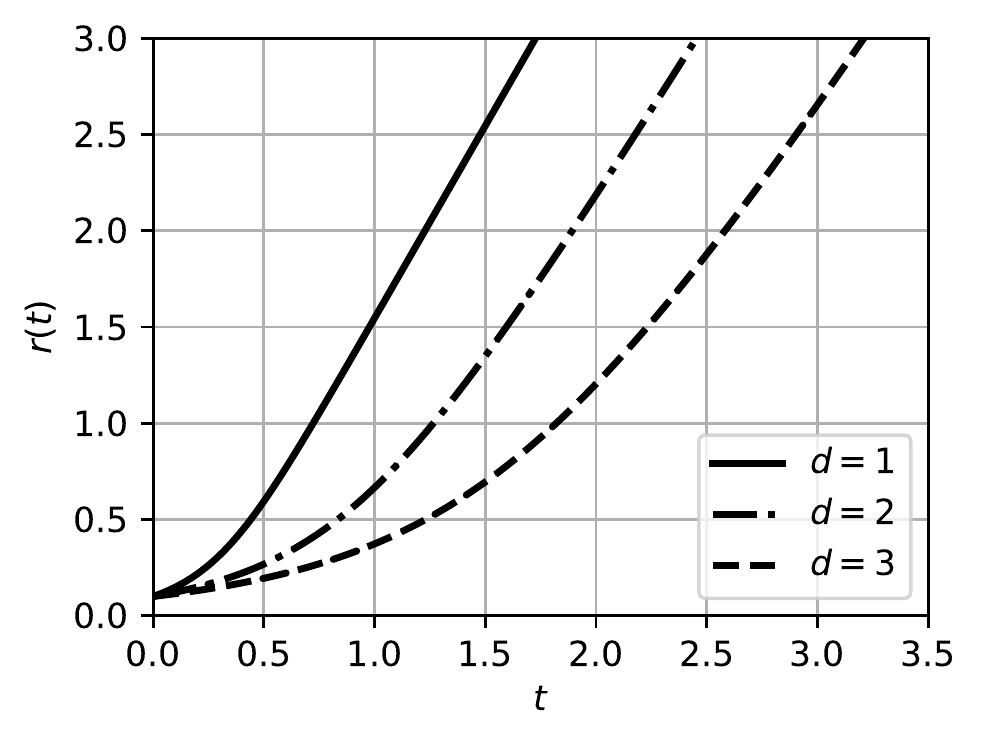}
	}}%
      \subcaptionbox{Pressure at origin vs time}[.3\linewidth]{
	\resizebox{.3\linewidth}{!}{
        \includegraphics[width=\textwidth]{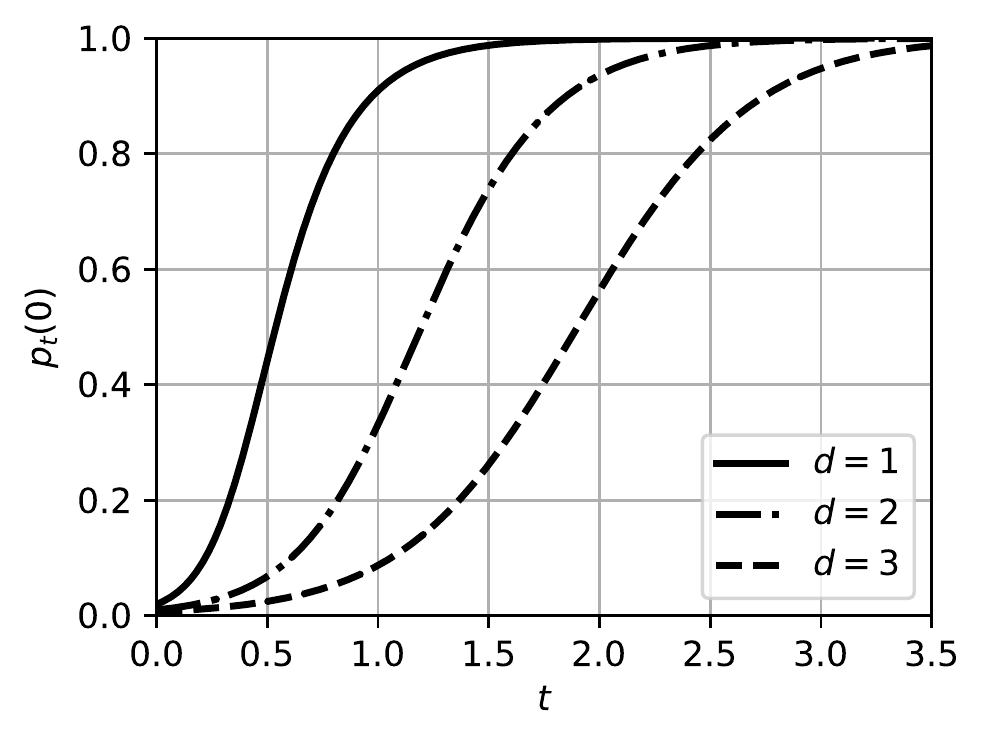}
	}}
	\subcaptionbox{Pressure vs distance from origin (fixed radius).}[.3\linewidth]{
	\resizebox{.3\linewidth}{!}{
        \includegraphics[width=\textwidth]{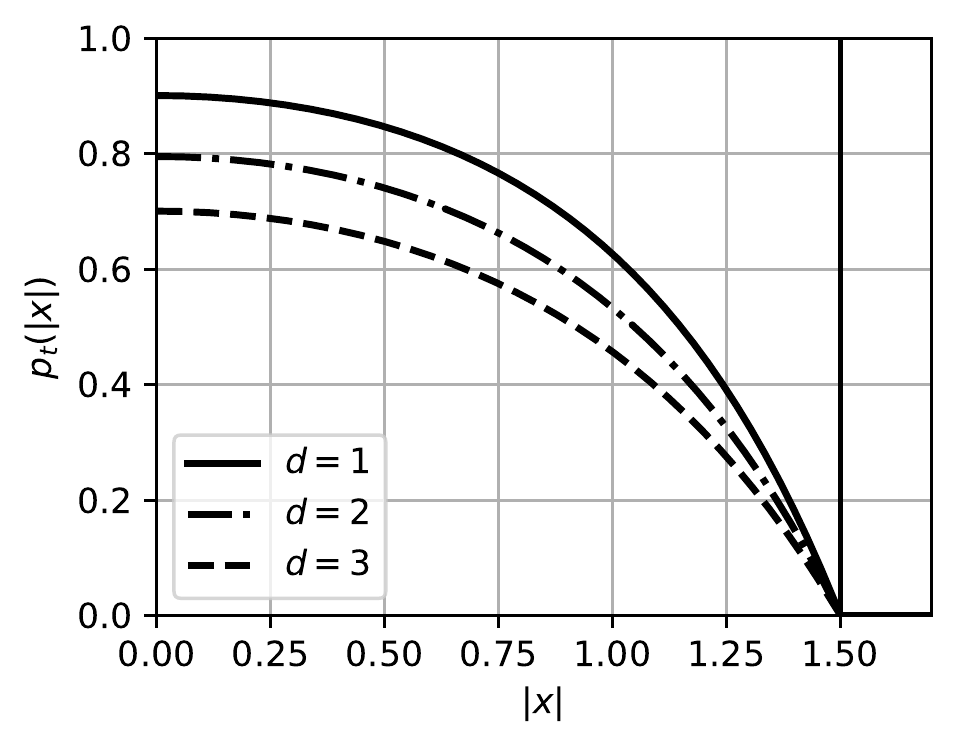}
	}}
	 \caption{Some properties of the spherical solutions, computed with the explicit formulae of Proposition \ref{prop:explicitball} in dimensions $d=1, 2 \text{ or } 3$. In all cases, $\lambda=1$, for (a)-(b) the initial condition is $r_0=0.1$ and for (c) the density radius is $r=1.5$.}
	 \label{fig:explicitball}
\end{figure}


\begin{proof}
Taking $\beta>0$, let us solve the evolution for the equation
$$ \partial \rho_t - \nabla \cdot ( \nabla p_t \rho_t ) = \beta^2 ( \lambda - p_t ) \rho_t \quad \text{and} \quad p_t(1-\rho_t)=0$$
(the Proposition is stated for $\beta=2$). In this case we suppose everything is radial, in particular we guess that $\rho_t= \chi_{B_{r(t)}}$ and also the pressure is radial. The pressure $p_t$ will depend only on $r(t)$ and it is the only function that satisfies
$$ \begin{cases}  -\Delta p = \beta^2(\lambda - p) \qquad & \text{ in }B_{r(t)} \\ p=0 & \text{ on }\partial B_{r(t)}. \end{cases}$$
Again, by symmetry we can suppose that $p_t(x)= f_t(|x|)$ where $f_t:\R_+ \to \R$ satisfies, with the expression of the Laplacian in spherical coordinates,
\[
\begin{cases}
f''_t(s) + \frac {d-1}{s} f'_t(s) - \beta^2 (f_t(s)- \lambda)=0 & \text{if $s\in [0,r(t)[$},\\
f_t(s) = 0 &\text{if $s\geq r(t)$} .
\end{cases}
\]
When we impose that $f_t'(0)=0$, if we consider $g=f-\lambda$ then we can see that the solution, assuming its smoothness and using Lemma \ref{lem:bessel}, is
$$g_t(s) = C_t H_{\alpha }(\beta s),$$
for some $C_t$ and $\alpha= \frac d2-1$. Then the condition $f_t(r(t))=0$ implies that $g_t(r(t))=-\lambda$ and this fixes $C_t=-\lambda / H_{\alpha}(\beta r(t))$. Now we have that $r'(t) = \frac{\partial p_t}{\partial n} = |f'_t(r(t))|$ and in particular we get:
\begin{align*}
r' &= |g'_t(r)| = -C_t \beta H'_{\alpha}(\beta r) = 
\lambda \beta  \frac {\beta r H_{\alpha+1 }(\beta r)}{H_{\alpha}(\beta r)} \\
&= \lambda \beta  \frac {I_{\alpha+1 }(\beta r)}{I_{\alpha}(\beta r)} = \lambda \beta  \frac {K_{\alpha+1 }(\beta r)}{ \beta r K_{\alpha}(\beta r)} 
= \lambda \beta^2  \frac {K_{\alpha+1 }(\beta r)}{ \beta K'_{\alpha+1}(\beta r)},
\end{align*}
and so we deduce that
$$\frac d {dt}  \log ( K_{\frac d2} (\beta r(t))   = \lambda \beta^2 $$
and thus
$$ K_{\frac d2} (\beta r(t) )= e^{ \lambda \beta^2 t } K_{\frac d2} (\beta r (0)).$$
Since for $\alpha>0$, $K_\alpha$ is strictly increasing and defines a bijection on $[0,+\infty[$, we have a well defined solution for \eqref{eq:mainPDE}. Uniqueness follows by Theorem \ref{maintheorem} or \cite{perthame2014hele} (since the initial density is of bounded variation).
\end{proof}

\subsection{Numerical results and comparison}
We now use the explicit solution for spherical tumor to assess the convergence of the numerical scheme when $\tau$ tends to zero. We fix an initial condition $\rho_0$ which is the indicator of a ball of radius $0.4$, we fix a final time $t_f=0.05$, and we observe the convergence towards the true solution of the continuous PDE \eqref{eq:mainPDE} when more and more intermediate time steps are taken (ie.\ as $\tau$ decreases). We perform the experiments in the 1-D and the 2-D cases and the results are displayed on Figure \ref{fig:numericball} with the following formulae:
\[
\text{rel. error on radius :} \frac{|r_\mathrm{num} -r_\mathrm{th}|}{r_\mathrm{th}} 
\quad \text{and} \quad
\text{rel. error on pressure :} \frac{\Vert p_\mathrm{num} - p_\mathrm{th} \Vert_\infty}{\Vert p_\mathrm{th} \Vert_\infty} 
\]
where the subscripts $\mathrm{``th"} $ and $\mathrm{``num"} $ refer to the theoretical and numerical computations. The theoretical pressure is compared to the numerical one on the points of the grid.

\paragraph{Dimension 1.} In the 1-D case, $\Omega = [0,1]$ is uniformly discretized into $N=4096$ cells $(W_i,x_i)=([(i-1)/N, i/N],(2i-1)/(2N))$ and $\epsilon=10^{-6}$. The results are displayed on Figure \ref{fig:numericball}(a)-(b), where the numerical radius is computed through $r_\mathrm{num} = (\sum_i \beta_i)/2N$.

\paragraph{Dimension 2.} In the 2-D case, $\Omega = [0,1]^2$ is uniformly discretized into $N^2$ sets $W_{i,j}=[(i-1)/N, i/N]\times [(j-1)/N, j/N]$ and $x_{i,j} = ((2i-1)/(2N),(2j-1)/(2N))$ with $N = 256$ and $\epsilon=2\times 10^{-5}$. Compared to the 1-D case, those parameter are less fine so that the computation can run in a few hours. The numerical radius is computed through $r_\mathrm{num}=\sqrt{\sum_i \beta_i/(\pi N^2)}$.

\paragraph{Comments}
We clearly observe the rate of convergence in $O(\tau)$ of the discretized scheme to the true solution. However, the \emph{locking} effect (mentioned in Section \ref{sss:discretization parameters}) starts being non-negligible for small values of $\tau$ in 2-D. This effect is even more visible on the 2-D pressure because the discretization is coarser. The pressure variables at $t=t_f$ for 2 different values of $\tau$ are displayed on Figure \ref{fig:pressurelocking}: the solution is more sensitive to the non-isotropy of the mesh when the time step $\tau$ is small. The use of random meshes could be useful to reduce this effect.

\begin{figure}
\centering
\subcaptionbox{Density (1d)}[.25\linewidth]{
 	\resizebox{.25\linewidth}{!}{
        \includegraphics{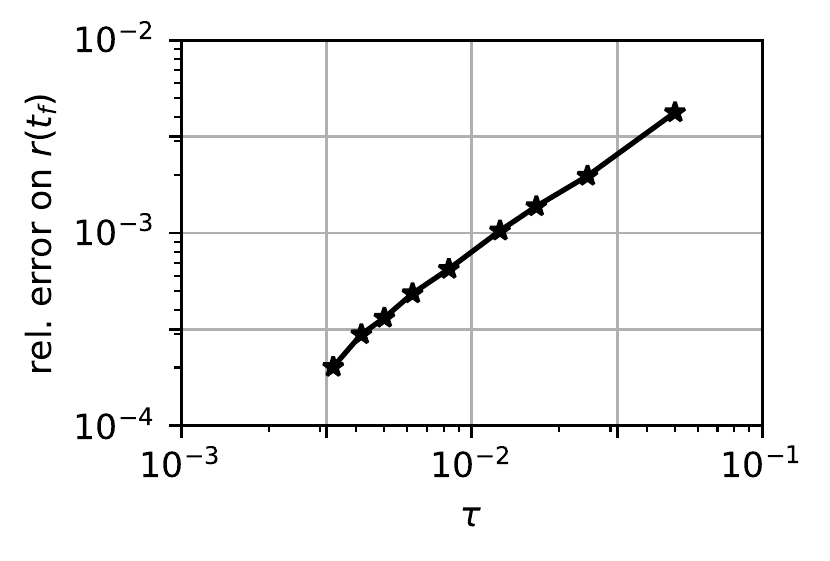}
	}}%
      \subcaptionbox{Pressure (1d)}[.25\linewidth]{
	\resizebox{.25\linewidth}{!}{
        \includegraphics[width=\textwidth]{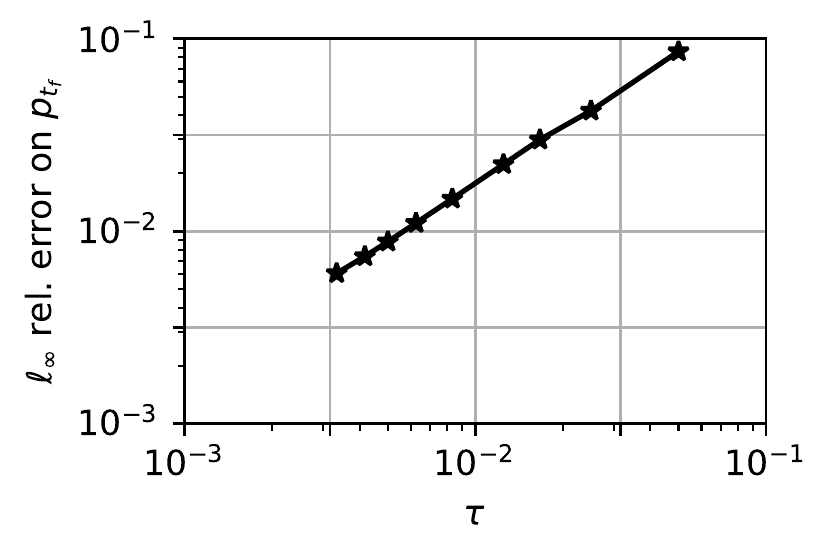}
	}}%
	\subcaptionbox{Density (2d)}[.25\linewidth]{
 	\resizebox{.25\linewidth}{!}{
        \includegraphics{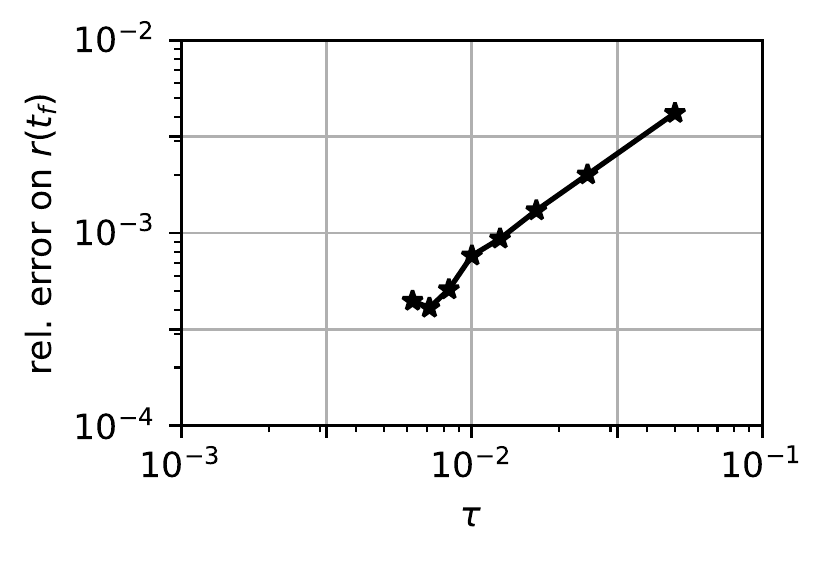}
	}}%
      \subcaptionbox{Pressure (2d)}[.25\linewidth]{
	\resizebox{.25\linewidth}{!}{
        \includegraphics[width=\textwidth]{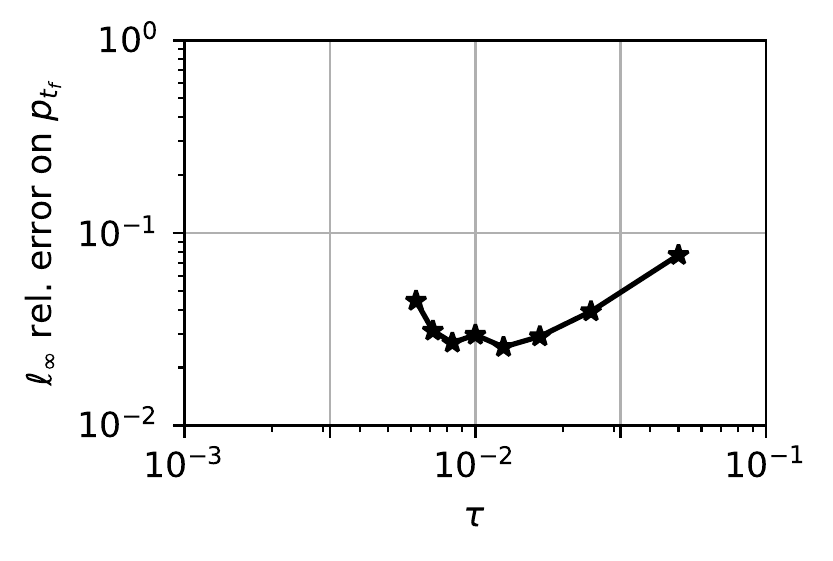}
	}}
	 \caption{For a fixed initial radius $r_0=0.4$ and final time $t_f=0.05$, we assess the convergence of the scheme as the time discretization $\tau$ tends to $0$ by comparing the computed $\rho_{t_f}$ and $p_{t_f}$ with the theoretical ones (see text body). Experiments performed on the interval $[0,1]$ discretized into $1024$ samples with $\epsilon=10^{-6}$ and on the square $[0,1]^2$ discretized into $256^2$ samples and $\epsilon=5.10^{-6}$.}
	 \label{fig:numericball}
\end{figure}

\begin{figure}
\centering
	\resizebox{.1\linewidth}{!}{
        \includegraphics[width=\textwidth]{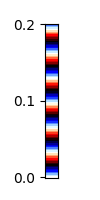}
	}%
	\subcaptionbox{theoretical (exact)}[.25\linewidth]{
 	\resizebox{.25\linewidth}{!}{
        \includegraphics{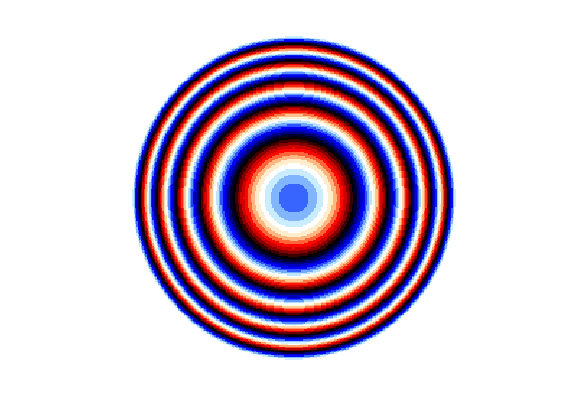}
	}}%
	\subcaptionbox{$\tau=2.5\times 10^{-2}$}[.25\linewidth]{
 	\resizebox{.25\linewidth}{!}{
        \includegraphics{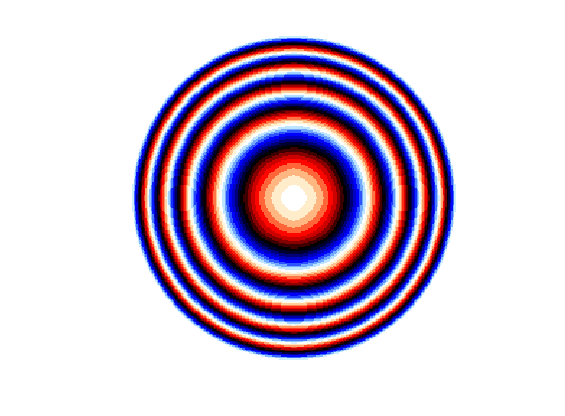}
	}}%
      \subcaptionbox{$\tau=6.25\times 10^{-3}$}[.25\linewidth]{
	\resizebox{.25\linewidth}{!}{
        \includegraphics[width=\textwidth]{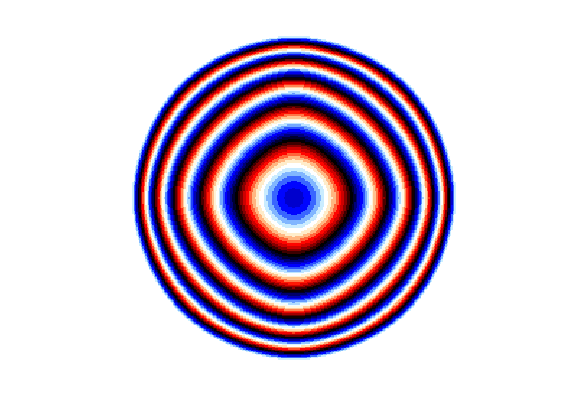}
	}}%
	 \caption{Pressure field at $t=t_f$, theoretical and numerical values for an initial density which is a ball in $[0,1]^2$. The pressure is a decreasing function of the distance to the center: here the colormap puts emphasis on the level sets so that the anisotropy due to the discretization is apparent.}
	 \label{fig:pressurelocking}
\end{figure}


\section{Illustrations}\label{sec:illustrations}
We conclude this article with a series of flows computed numerically.

\subsection{On a 1-D domain}
We consider a measure $\rho_0$ of density bounded by 1 on the domain $[0,1]$ discretized into $1024$ samples (Figure \ref{fig:flow1d}-(a), darkest shade of blue) and compute the evolution of the flow with parameters $\tau = 10^{-2}$ and $\epsilon=1e-6$. The density at every fourth step is shown with colors ranging from blue to yellow as time increases. Each density is displayed behind the previous ones, without loss of information since density is non-decreasing with time. The numerical pressure is displayed on Figure \ref{fig:flow1d}-(b).

\begin{figure}
\centering
\subcaptionbox{Density}[\linewidth]{
        \includegraphics[width=0.75\textwidth]{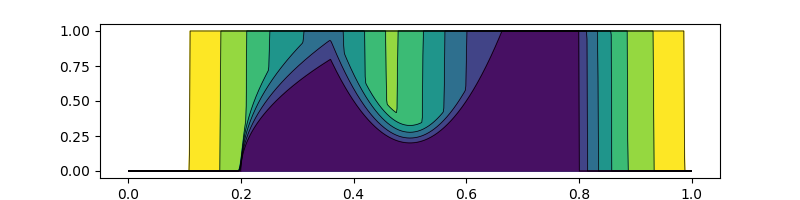}
}
      \subcaptionbox{Pressure}[\linewidth]{
        \includegraphics[width=0.75\textwidth]{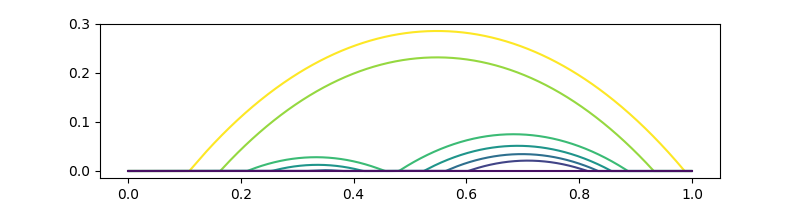}
}
	 \caption{A flow on $[0,1]$ and the associated pressure. Time shown are $t=0, 0.04, 0.08, \dots, 0.24$, color ranges from blue to yellow as time increases}
	 \label{fig:flow1d}
\end{figure}

\paragraph{Splitting scheme}
We also compare this evolution with a splitting scheme, inspired by~\cite{gallouet2016jko, gallouet2017unbalanced}, that allows for a greater freedom in the choice of the function $\Phi$ that relates the pressure to the rate of growth in \eqref{eq:mechanicalmodel}. This scheme alternates implicit steps with respect to the Wasserstein metric and the Hellinger metric and is as follows. Let $\rho_0 \in \mathcal{M}_+(\Omega)$ be such that $\rho_0 \leq 1$ and define, for $n\in \mathbb{N}$,
\[
\begin{cases}
\rho^\tau_{2n+1} = P^{W_2}(\rho^\tau_{2n})\\
\rho^\tau_{2n+2}(x) = \rho^\tau_{2n+1} /(1-\tau\Phi(p^\tau_{2n+1}(x))) \text{ for all $x\in \Omega$},
\end{cases}
\]
where $P^{W_2}$ is the projection on the set of densities bounded by $1$ for the Wasserstein distance and $p_n$ is the pressure field corresponding to the projection of $\rho^\tau_n$ (as in~\cite{de2016bv}). The degenerate functional we consider is outside of the domain of  validity of the results of~\cite{gallouet2016jko, gallouet2017unbalanced} and we do not know how to prove the convergence of this scheme for non linear $\Phi$. It is this introduced here nerely for informal comparison with the case $\Phi$ linear. 

It is rather simple to adapt Algorithm \ref{alg:mainalgo} to compute Wasserstein projection on the set of measures of density bounded by 1. 
On Figure \ref{fig:splitting1d}, we display such flows for rates of growth of the form $\Phi(p)=4(1-p)^\kappa$, for three different values of $\kappa$. For these computations, the segment $[0,1]$ is divided into $1024$ samples, $\tau=10^{-2}$, $\epsilon=10^{-6}$ and we display the density after the projection step, at the same times than on Figure \ref{fig:flow1d}. With $\kappa=1$, we should recover the same evolution than on Figure \ref{fig:flow1d}. As $\kappa$ increases, the rate of growth is smaller when the pressure is positive, as can be observed on Figures \ref{fig:splitting1d}-(a-c).

\begin{figure}
\centering
\subcaptionbox{$\kappa = 0.01$}[0.33\linewidth]{
        \includegraphics[width=0.35\textwidth]{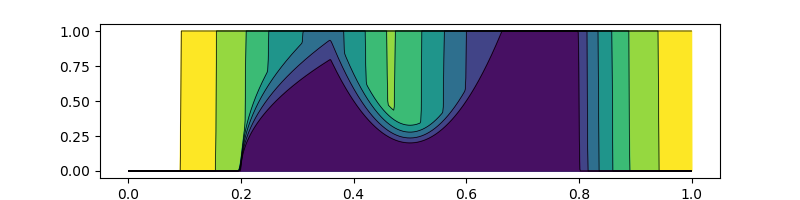}
}%
      \subcaptionbox{$\kappa = 1$}[0.33\linewidth]{
        \includegraphics[width=0.35\textwidth]{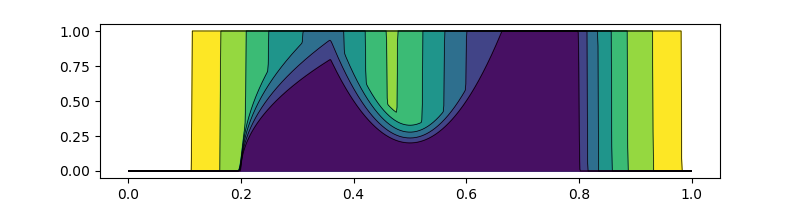}
}%
\subcaptionbox{$\kappa = 10$}[0.33\linewidth]{
        \includegraphics[width=0.35\textwidth]{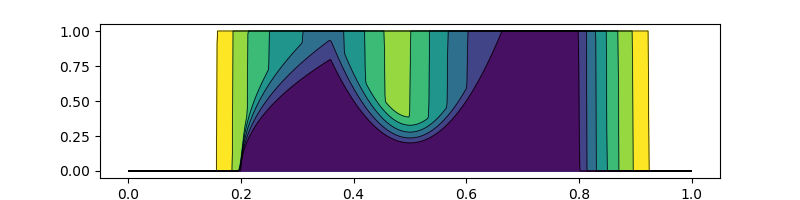}
}
	 \caption{Flow on the line with the implicit splitting scheme, with rate of growth of the form $\Phi(p)=4(1-p)^\kappa$. }
	 \label{fig:splitting1d}
\end{figure}

\subsection{On a 2d non-convex domain}
Our last illustration is performed on the square $[0,1]^2$ discretized into $256^2$ samples. The initial density $\rho_0$ is the indicator of a set and the parameters are $\tau=0.015$ and $\epsilon=5.10^{-6}$. The first row of Figure \ref{fig:flow2d} shows the flow at every 10th step (equivalent to a time interval of $0.15$). Except at its frontier (because of discretization), the density remains the indicator of a set at all time. The bottom row of Figure \ref{fig:flow2d} displays the pressure field, with a colormap that puts emphasis on the level sets. Notice that its level sets are orthogonal to the boundaries.

\begin{figure}
\centering
\begin{tikzpicture}[scale=1.0]
\pgfmathsetmacro{\ax}{2.15}
\pgfmathsetmacro{\ay}{1.5}
\pgfmathsetmacro{\b}{0.0}
\pgfmathsetmacro{\s}{0.23}
\node (1) at (0*\ax,\ay) {\includegraphics[clip,trim=\b cm \b cm \b cm \b cm, scale=\s]{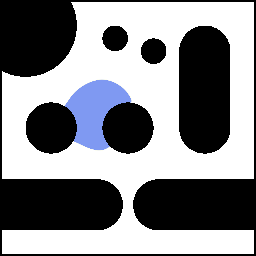}};
\node (2) at (1*\ax,\ay) {\includegraphics[clip,trim=\b cm \b cm \b cm \b cm, scale=\s]{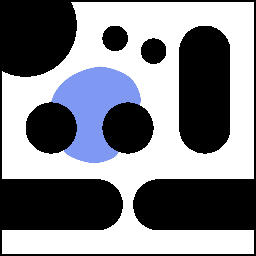}};
\node (3) at (2*\ax,\ay) {\includegraphics[clip,trim=\b cm \b cm \b cm \b cm, scale=\s]{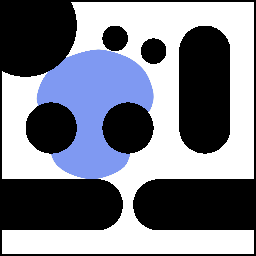}};
\node (4) at (3*\ax,\ay) {\includegraphics[clip,trim=\b cm \b cm \b cm \b cm, scale=\s]{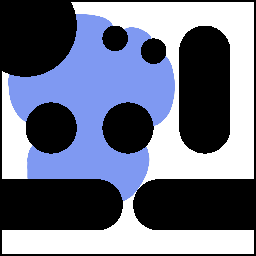}};
\node (5) at (4*\ax,\ay) {\includegraphics[clip,trim=\b cm \b cm \b cm \b cm, scale=\s]{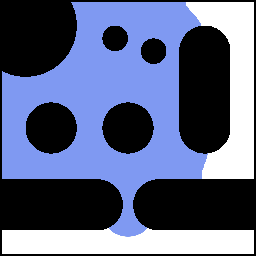}};
\node (6) at (5*\ax,\ay) {\includegraphics[clip,trim=\b cm \b cm \b cm \b cm, scale=\s]{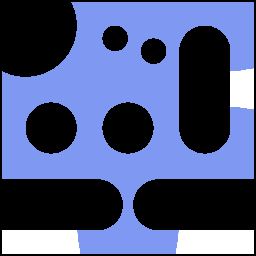}};
\node (7) at (6*\ax,\ay) {\includegraphics[clip,trim=\b cm \b cm \b cm \b cm, scale=\s]{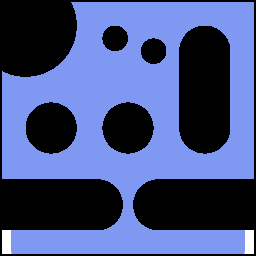}};

\node (1) at (0*\ax,-\ay+.2) {\includegraphics[clip,trim=\b cm \b cm \b cm \b cm, scale=.36]{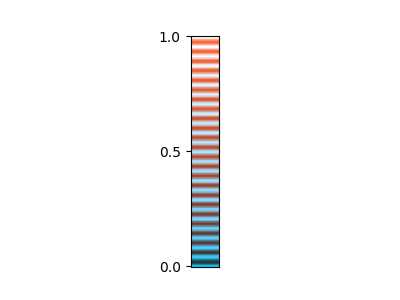}};
\node (2) at (1*\ax,-\ay) {\includegraphics[clip,trim=\b cm \b cm \b cm \b cm, scale=\s]{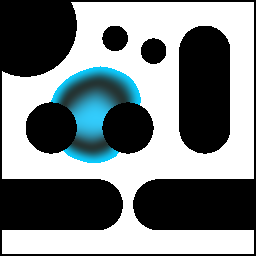}};
\node (3) at (2*\ax,-\ay) {\includegraphics[clip,trim=\b cm \b cm \b cm \b cm, scale=\s]{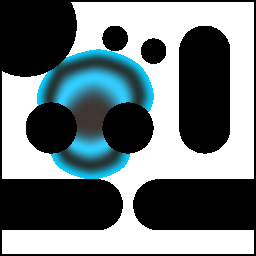}};

\node (4) at (3*\ax,-\ay) {\includegraphics[clip,trim=\b cm \b cm \b cm \b cm, scale=\s]{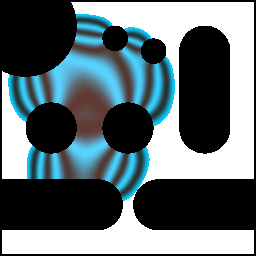}};
\node (5) at (4*\ax,-\ay) {\includegraphics[clip,trim=\b cm \b cm \b cm \b cm, scale=\s]{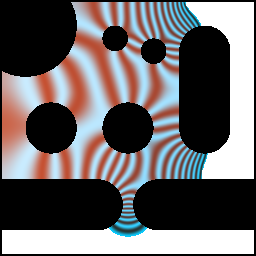}};
\node (6) at (5*\ax,-\ay) {\includegraphics[clip,trim=\b cm \b cm \b cm \b cm, scale=\s]{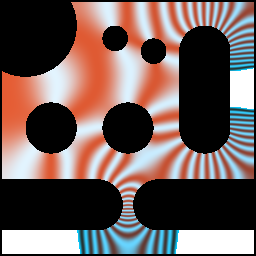}};
\node (7) at (6*\ax,-\ay) {\includegraphics[clip,trim=\b cm \b cm \b cm \b cm, scale=\s]{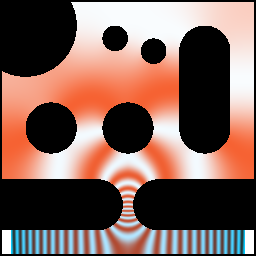}};

\draw[thick,->] 	(0,0.1)    	-- 	(6*\ax,0) node[midway,below]{$t$};

\end{tikzpicture}
	 \caption{Evolution on a non convex 2-d domain (obstacles in black). (top row) evolution of the density, the colormap is linear from white to blue as the density goes from $0$ to $1$. (bottom row) pressure represented with a striped colormap to make the level sets apparent. White area corresponds to $p=\rho=0$.}
	 \label{fig:flow2d}
\end{figure}
\appendix
\section{Appendix}
\subsection{Tools of measure theory}

\begin{lemma} Let $(X, d, \bar{\mu})$ be a metric measure space with finite measure. Let $\mu_n=f_n \bar{\mu}$ and $\mu= f \bar{\mu}$. Let us suppose that $\mu_n \weakto \mu$ and $\Ent{\mu_n}{\bar{\mu}}\to \Ent{ \mu}{\bar{\mu}} $. Moreover let us assume there exists maps $T_n : X \to \mathbb{R}^d$ such that $(Id, T_n)_{\#} \mu_n \weakto (Id, T)_{\#} \mu$, with $T$ bounded. Then we have, up to a subsequence
\begin{itemize}
\item[(i)] \label{lem:conv1} $f_n \to f$  in  $L^1(\bar{\mu})$;
\item[(ii)]  \label{lem:conv2}  $T_n(x) \to T(x)$ for $\mu$-a.e. $x \in X$.
\end{itemize}
\end{lemma}

\begin{proof} The first point is a well known consequence of the strict convexity of $t \log t$ (see for example \cite[Theorem 3]{visintin1984strong}) and the fact that since $f_n$ are uniformly $\mu$ integrable we have $f_n\weakto f$ in $L^1(\bar{\mu})$. For the second point it is sufficient to notice that thanks to the first point we have $(Id, T_n)_{\#} \mu \weakto (Id, T)_{\#} \mu$ and then we can apply \cite[Lemma 5.4.1]{ambrosio2008gradient} and then pass to a subsequence.
%
%
%
%
\end{proof}

\subsection{Technical lemmas}

\begin{proof}[Proof of Lemma \ref{lem:f_n}] It is easy to compute the first and second derivative of $f$ to deduce that we have also $f''(t)=2e^f \geq 2$. Now we consider and increasing sequence of points $t_n \leq \pi/2$ such that $t_n  > 1$ and $t_n \uparrow \pi/2$; then we define functions $f_n$ such that $f_n(0)=0$, $f_n'(0)=0$ and
\begin{equation}\label{eqn:def_fn}
f''_n(t) = \begin{cases} 2e^{f(t)} \qquad &\text{ if }t \leq t_n \\ e^{-t} & \text{ otherwise.} \end{cases}
\end{equation}
Since $f_n'' >0$ uniformly on bounded sets we have that $f_n$ is strictly convex; moreover it is Lipschitz since 
\begin{align*}
f_n'(t) &=f_n'(t)-f_n'(0)= \int_0^{t} f_n''(s)\, ds \\
&\leq \int_0^{\infty} f_n''(s)\, ds=\int_0^{t_n} f''(s) \,ds + \int_{t_n}^{\infty} e^{-s} \, ds \\
& \leq f'(t_n) + 1.
\end{align*}

Furthermore clearly since $t_n$ is increasing we have that $f''_n$ is an increasing sequence of functions, in fact $f''(t) > e^{-t}$. Moreover it is clear that $f_n(t)=f(t)$ for $t \in [0,t_n]$ (and in particular in $[0,1]$), and so we have $f_n \uparrow f$ in $[0, \pi/2)$ but, since $f_n$ are increasing functions in $t$, we conclude also that $f_n (t) \to +\infty$ for every $t \geq \pi/2$.

As for (ii) we denote $F(t)=f_n'(t)^2$ and $G(t)=4(e^{f_n(t)} -1)$. First we notice that $F(t)=G(t)$ for $t \in [0,t_n]$, since here $f_n$ agrees with $f$, that satisfies the differential equation; then, for $t>t_n$, we can apply the Cauchy's mean value theorem to $F$ and $G$, that are both strictly increasing and differentiable in $(t_n, \infty)$. In particular there exists $t_n<s<t$ such that
$$ \frac{ F(t) - F(t_n) }{G(t) - G(t_n)}  = \frac { F'(s)}{G'(s)}= \frac{ 2f_n'(s) f_n''(s)}{ 4 f_n'(s) e^{f_n(s)}} = \frac {f_n''(s)}{2e^{f_n(s)}} \leq e^{-s} \leq 1; $$
knowing that $F(t_n)=G(t_n)$ and that $G(t) >G(t_n)$ we get immediately that $F(t) \leq G(t)$.

For the second inequality we will use that $e^t-1 \geq t $ and $e^t-1 \geq t^2/2$. We choose $t_n$ big enough such that $f(t_n)/t_n \geq \sqrt{2}$: this is always possible since $f(t)/t \to \infty$ as $ t \uparrow \pi/2$. Then from equation \eqref{eqn:def_fn} we have $f_n''(t) \geq 2 $ for $t \leq t_n$ and in particular $f_n(t) \geq t^2$ in that region and so we get
$$ e^{f_n(t)}-1 \geq e^{t^2}-1 \geq t^2 \qquad \forall t \leq t_n,$$
while if $t \geq t_n$ by convexity we have $f_n(t) \geq \frac { f_n(t_n)}{t_n} t \geq \sqrt{2} t$ and so 
$$ e^{f_n(t)}-1 \geq e^{\sqrt{2}t}-1 \geq \frac 12 (\sqrt{2} t)^2=t^2 \qquad \forall t \geq t_n,$$
concluding thus the proof.
 \end{proof}
 
 \begin{lemma} \label{lem:mass} Let us consider $\mu_1, \mu_2$ two measures in $\Omega$ and a Borel cost $c \geq 0$. It holds
$$ T_c(\mu_1,\mu_2) \geq \Bigl(\sqrt{ \mu_1(\Omega) } - \sqrt{ \mu_2(\Omega) }\Bigr)^2  .$$
\end{lemma}

\begin{proof} In the sequel, we denote $\mu_i(\Omega)=m_i$. It is clear that we can suppose $c =0$ (in fact $T_{c'} \leq T_{c}$ whenever $c' \leq c$) and write our problem as
\begin{align*}
T_0(\mu_1, \mu_2) &= \min_{\gamma \in \mathcal{M}_+( \Omega \times \Omega )} \left\{ \Ent { \gamma_1}{ \mu_1} + \Ent{ \gamma_2} {\mu_2}  \right\} \\
& = \min_{M \geq 0} \min_{\gamma \in \mathcal{M}_+( \Omega \times \Omega )} \left\{ \Ent { \gamma_1}{ \mu_1} + \Ent{ \gamma_2} {\mu_2} \; : \;  m(\gamma)=M \right\}\, .
\end{align*}
We can restrict ourselves to the case $\gamma_i \ll \mu_i$, where we have $\gamma_i= \sigma_i \mu_i$ and using Jensen inequality applied to $E(t) \eqdef t \ln t -t +1$ it holds
\begin{align*} \Ent{\gamma_i}{\mu_i} &= m_i \int_{\R^d} E( \sigma_i ) \, d \frac {\mu_i}{m_i} \geq m_i E\left( \int_{\R^d} \sigma_i d \frac {\mu_i}{m_i} \right) \\ &= m_i E \left( \frac M{m_i} \right) = M \ln \left( M/ m_i \right)  - M + m_i,\end{align*}
 with equality if we choose $\gamma_i = \frac M{m_i}\mu_i $ and $\gamma= \gamma_1 \otimes \gamma_2$. In particular, we have
$$
T_0 (\mu_1, \mu_2) = \min_{M \geq 0}  \left\{ M \ln \left( \frac {M^2}{m_1m_2} \right) + m_1+m_2 -2M \right\}\, ,
$$
the minimizer is $M= \sqrt{ m_1 m_2}$, so
$T_c(\mu_1, \mu_2) \geq T_0(\mu_1, \mu_2) =( \sqrt{ m_1} - \sqrt{ m_2} )^2$.
\end{proof}

 \subsection{Explicit form of geodesics and convexity}
 
 \begin{theorem}\label{thm:rep} Let us consider two absolutely continuous measures $\mu_0$ and $\mu_1$. Then, given $\phi$ an optimal potential for the problem $T_{c_l} (\mu_0, \mu_1)$ we consider the quantities 
 $$\begin{cases} \alpha_t(x) & = (1-\phi(x) t)^2 + \frac{t^2 | \nabla \phi(x)|^2}4 \\ X_t(x) &= x- {\rm arctan } \bigl( \frac { t|\nabla \phi (x)| } {2-2t\phi(x)}\bigr) \frac { \nabla \phi (x) } {| \nabla \phi (x) |}.\end{cases}$$
 Then we have that $\mu_t =(X_t)_{\#} ( \alpha_t \mu_0)$ is the geodesic for $\dist$ between $\mu_0$ and $\mu_1$.
 \end{theorem}
 
 We first perform a formal proof starting from the geodesics equations, that can be useful in other cases, when the explicit form of the geodesics is not known. This will be followed by the correct proof, that uses the cone construction introduced \cite{liero2015optimal} in order to justify everything without the need to take more than one derivative of the potential.
 
 \begin{proof}[Formal proof] Let us consider the equation of the geodesic:
 $$ \begin{cases} \partial_t \mu_t  + \nabla \cdot ( \nabla \xi_t \mu_t ) & = 4 \xi_t \mu_t  \\ \partial_t \xi_t + \frac 12 | \nabla \xi_t |^2 + 2 \xi_t^2 & = 0. \end{cases} $$
 We know that, letting $X_t$ be the flow of $\nabla \xi_t$ we have that a possible solution to the first equation is $\mu_t = (X_t)_{\#} ( \alpha_t \mu_0)$, where $\alpha_t = e^{ \int_0^r 4 \xi_r (X_r) \, dr} $ (by direct computation). So we want to solve
 $$ \begin{cases} \frac d{dt} X_t (x) &= \nabla \xi_t (X_t(x))  \\ X_0(x)&=x. \end{cases} $$
 Next we compute formally $\frac d{dt} \xi_t (X_t(x)) $ and $\frac d{dt} \nabla \xi_t (X_t(x)) $:
 \begin{align*} \frac d{dt} \xi_t (X_t(x)) &= - \frac 12 | \nabla \xi_t|^2 (X_t) - 2 \xi_t^2 (X_t) + \nabla \xi_t \cdot \frac d{dt} X_t(x) \\
 & = \frac 12 | \nabla \xi_t|^2 (X_t) - 2 \xi_t^2 (X_t);
 \end{align*}
 \begin{align*} \frac d{dt} \nabla \xi_t (X_t(x)) &= - D^2 \xi_t \cdot \nabla \xi_t (X_t) - 4 \xi_t \nabla \xi_t (X_t) + D^2 \xi_t \cdot \frac d{dt} X_t(x) \\
 & = -4 \xi_t(X_t) \nabla \xi_t(X_t).
 \end{align*}
 From the second equation we get that $\nabla \xi_t (X_t) = \nabla \xi_0 (x) /\alpha_t$. Furthermore, denoting $\gamma_t = \xi_t (X_t(x))$ we can write a system of differential equations for $\alpha_t$ and $\gamma_t$:
 $$ \begin{cases} \dot{ \gamma_t } &= \frac { | \nabla \xi_0 (x)|^2 }{2 \alpha_t^2} - 2\gamma_t^2 \\
 \dot{\alpha_t} & = 4 \gamma_t \alpha_t.
 \end{cases}$$
We can now substitute $\gamma_t = \dot{ \alpha}_t / 4\alpha_t$, getting
$$\frac{ \ddot{ \alpha }_t}{4 \alpha_t} - \frac{ \dot{ \alpha }^2_t}{4 \alpha^2_t} =  \dot{ \gamma_t } = \frac { | \nabla \xi_0 (x)|^2 }{2 \alpha_t^2} - \frac{\dot{ \alpha }^2_t}{8 \alpha^2_t}$$
$$ 2 \alpha_t \ddot{\alpha}_t = 4 | \nabla \xi_0(x)|^2 + \dot{\alpha}_t^2.$$
 Now it is easy to see that the solution to the last equation is a quadratic polynomial $\alpha_t = at^2+bt + c$. We know that $c= \alpha_0=1$, while $b= \dot{\alpha}_0 = 4 \xi_0 (x)$. The equation gives then $4a= b^2+ 4 | \nabla xi_0(x)|^2$ and thus $a= |\nabla \xi_0(x)|^2 + 4 | \xi_0(x)|^2$. Concluding we get precisely
 $$ \alpha_t (x)= (1+2t\xi_0(x))^2 + ( t |\nabla \xi_0 (x)|)^2$$ 
 $$ X_t (x) = x+ \nabla \xi_0 (x) \int_0^t \frac 1{\alpha_r(x)} \, dr = x+  \frac{\nabla \xi_0 (x)}{|\nabla \xi_0|(x)} {\rm arctan} \left( \frac{t |\nabla \xi_0 (x)|}{1+2t\xi_0(x)} \right). $$
Now we simply use the fact that $\phi(x) = - 2\xi_0 (x)$ is a good potential to conclude.
 \end{proof}

 \begin{proof} Let us consider $Y_t(x,r) = ( r \sqrt{ \alpha_t(x) }, X_t(x) )$: then $Y_t$ are geodesics in the cone. By the cone construction to have that if $\mu_0, \mu_1$ are two measures on $\R^d$ and $\phi$ is an optimal potential for $T_c(\mu_0, \mu_1)$ we have that $\phi(x)r^2$ is an optimal potential for $\nu_0 (x,r)= \mu_0(x) f(r), \nu_1= (Y_1)_{\#} (\nu_0)$ for any $f$ such that $\int f(r)r^2 \, dr=1$. Notice that $\mathfrak{P} \nu_0= \mu_0$, $\mathfrak{P} \nu_1= \mu_1$ and moreover $W_2(\nu_0, \nu_1) = \dist (\mu_0, \mu_1)$; in particular since $\nu_t= (Y_t)_{\#} (\nu_0)$ is a geodesic for $W_2$ on the cone, we will have that $\mu_t=\mathfrak{P} \nu_t $ is the geodesic for $\dist$.
 \end{proof}
 
 \begin{lemma}\label{lem:derivativemut} Let $\mu_0, \mu_1$ be two absolutely continuous measures on $\R^d$  such that $\mu_0, \mu_1 \leq 1$ and let us consider $f \in H^1(\R^d)$. Then, if we consider $\mu_t$ the geodesic between $\mu_0$ and $\mu_1$, we have 
 $$ \frac { d}{dt} \Bigg|_{t=0}\int_{\R^d} f \, d \mu_t   = - \int_{\R^d} (2 f \phi + \frac 12 \nabla f \cdot \nabla \phi ) \, d \mu_0 $$
 \end{lemma}
 
 \begin{proof} Using Theorem \ref{thm:rep} we can write explicitly
 $$ \int f \, d \mu_t =  \int_{\R^d} f( X_t(x)) \alpha_t(x) \, d \mu_t.$$
 Now we can use that $\frac d{dt} X_t= -\nabla \phi (x)/ 2\alpha_t$ in order to get
 $$ \frac { d}{dt} \Bigg|_{t=0}  \int f \, d \mu_t  = - \frac 12 \int_{\R^d} \nabla f (X_t) \cdot  \nabla \phi \, d \mu_0 + \int_{\R^d} f(X_t) \Bigl( \frac {t| \nabla \phi|^2}2 - 2 (1-t\phi) \phi\Bigl) \, d \mu_0. $$
 While this calculation is clear when $f \in C^{\infty}_c$ in order to make sense for $f \in H^1$ we have to consider the finite difference and integrate this inequality:
 \begin{align*} \frac{\int f(x) \, d \mu_t - \int f(x) \, d \mu_0 }t 
 &= \frac 1t \int_{\R^d} \int_0^t \frac d {dt} f(X_t(x)) \alpha_t(x) |_{t=s} \, ds \, d \mu_0    \\
 &=  \frac 1t \int_{\R^d} \int_0^t - \frac 12 \nabla f (X_s) \nabla \phi(x)  \\
 &\quad +f(X_s) \Bigl( \frac {s| \nabla \phi|^2}2 - 2 (1-s\phi) \phi\Bigl) \, ds \, d \mu_0 \\
 & = \int_{\R^d} - \frac 12 \mathcal{A}_t( \nabla f ) \cdot \nabla \phi + \mathcal{C}_t ( f ) | \nabla \phi| -2\mathcal{B}_t ( f ) \phi  \, d \mu_0,
 \end{align*}
 where we denoted by $ \mathcal{A}_t, \mathcal{B}_t,\mathcal{C}_t$ three linear operator which we will show that are acting continuously from $L^2(\R^d)$ to $L^2( \mu_0)$, thus proving the formula for $f \in H^1( \R^d)$. Explicitly we have
 $$ \mathcal{A}_t (g ) (x) = \frac 1t \int_0^t g(X_s(x)) \, ds \qquad \mathcal{B}_t (g ) (x) = \frac 1t \int_0^t g(X_s(x)) (1-s \phi(x) ) \, ds $$
 $$ \mathcal{C}_t (g ) (x) = \frac 1t \int_0^t g(X_s(x))\frac{s|\nabla \phi (x)|}2 \, dt. $$

 Notice that for $0< s \leq t<1$ we have always $1-s\phi \geq 1-t$. Now using Jensen and then Fubini we get:
 \begin{align*}
 \int |\mathcal{A}_t (g ) (x)|^2 \, d \mu_0 & \leq \frac 1t \int_0^t \int g (X_s)^2 \, d \mu_0 \, ds \\
 &\leq  \frac 1 { (1-t)^2} \cdot \frac 1t \int_0^t \int g (X_s)^2 \alpha_s \, d \mu_0  \, ds \\
 & = \frac 1 { (1-t)^2} \cdot  \frac 1t \int_0^t \int g (x)^2 \, d \mu_s \, ds
 \end{align*} 
  \begin{align*}
 \int |\mathcal{B}_t (g ) (x)|^2+ |\mathcal{C}_t(g) (x)|^2 \, d \mu_0 & \leq \frac 1t \int_0^t \int g (X_s)^2 \left((1-s \phi(x))^2 +  \frac{s^2|\nabla \phi (x)|^2}4 \right)  \, d \mu_0 \, ds \\
 &= \frac 1t \int_0^t \int g (X_s)^2 \alpha_s \, d \mu_0  \, ds \\
 & = \frac 1t \int_0^t \int g (x)^2 \, d \mu_s \, ds.
 \end{align*}
 We can thus conclude thanks to the fact that $\mu_s \leq 1$, by Theorem \ref{thm:convexity}. In particular we have $ \| \mathcal{A}_t \| \leq 1/(1-t)$ and $\| \mathcal{B}_t\| , \| \mathcal{C}_t\| \leq 1$. Now we only need to show that $\mathcal{A}_t g \to g$, $\mathcal{B}_t g \to g$, $\mathcal{C}_t g \to 0$, where all these convergences are to be considered strongly in $L^2$. Thanks to the fact that these operators are bounded it is sufficient to show that this is true for a dense set of functions. But for $g \in C^{\infty}_c$ for $s < 1/2$ we have $|g(X_s) - g(x)| \leq L \cdot  {\rm arctan}(s| \nabla \phi  (x)|)$ where $L$ is the Lipschitz constant of $g$. Then we have (using that $\phi, \nabla \phi \in L^2(\mu_0)$)
 $$ \int |\mathcal{A}_t (g)(x) - g(x)|^2 \, d \mu_0  \leq {L^2} \int {\rm arctan} ( t|\nabla \phi| )^2 \, d \mu_0 \to 0,$$
$$\int |\mathcal{B}_t (g) (x)- \mathcal{A}_t(g) (X) |^2 \, d \mu_0 \leq  \frac{\| g \|_{\infty}}4 \int t^2 \phi^2 \, d \mu_0 \to 0 $$
$$ \int |\mathcal{C}_t (g) (x) |^2 \, d \mu_0 \leq \frac{\| g \|_{\infty}}{16} \int t^2 |\nabla \phi|^2 \, d \mu_0 \to 0. $$

In particular we proved that, for every $f \in H^1(\R^d)$ we have
 \[
  \frac { d}{dt} \int_{\R^d} f \, d \mu_t  |_{t=0} = - \int_{\R^d} (\frac 12 \nabla f \cdot \nabla \phi  + 2f \phi) \, d \mu_0. \qedhere
  \]
 \end{proof}

\begin{lemma}\label{lem:appendix} Let $\mu_0, \mu_1$ be two absolutely continuous measures on $\Omega$ convex such that $\mu_0, \mu_1 \leq 1$ and let us consider $p \in H^1(\Omega)$, such that $p \geq 0$ and $p (1- \mu_0)=0$. Then, if we consider $\phi$ an optimal potential between $\mu_0$ and $\mu_1$, we have
 $$  \int_{\R^d} (2 p \phi + \frac 12 \nabla p \cdot \nabla \phi ) \, d \mu_0  \geq 0.$$
 \end{lemma}
\begin{proof} Let us consider $\mu_t$ the geodesic between $\mu_0$ and $\mu_t$. We know that $\mu_t \leq 1$ by Theorem \ref{thm:convexity} and $\mu_t$ will be supported on $\Omega$ as well. In particular we have $ \int_{\Omega} p \, d \mu_t  \leq \int_{\Omega} p = \int_{\Omega} \, d \mu_0$. But then using Lemma \ref{lem:derivativemut} it is easy to conclude
\[
\int_{\R^d} (2 p \phi + \frac 12 \nabla p \cdot \nabla \phi ) \, d \mu_0 = - \frac { d}{dt} \Bigg|_{t=0}  \int p (x) \, d \mu_t \geq 0.\qedhere
\]
\end{proof}

\bibliographystyle{plain}
\bibliography{refs}

\end{document}